\documentclass[a4paper,twoside,11pt]{article}

\usepackage{a4wide} 
\usepackage[utf8]{inputenc}
\usepackage[T1]{fontenc}
\usepackage{lmodern}

\usepackage{graphicx} 
\usepackage{verbatim} 
\usepackage{amsmath,amssymb,amsthm}
\usepackage{color}
\usepackage{enumerate}
\usepackage{array}
\usepackage[noadjust]{cite}
\usepackage{multirow}

\usepackage{caption}
\usepackage{subcaption}

\usepackage{xcolor}
\usepackage{hyperref}
\definecolor{darkgreen}{rgb}{0,0.4,0}
\definecolor{BrickRed}{rgb}{0.65,0.08,0}
\hypersetup{colorlinks=true,linkcolor=blue,citecolor=red,filecolor=BrickRed,urlcolor=darkgreen}
\newcommand{\ct}{c}
\newcommand{\rt}{r}

\newcommand{\LandauO}{\mathcal{O}}
\newcommand{\Landauo}{o}
\newcommand{\PR}{\mathbb{P}}

\newcommand{\Cc}{\mathcal{C}}

\newcommand{\Lc}{\mathcal{L}}

\newcommand{\N}{\mathbb{N}}
\newcommand{\Q}{\mathbb{Q}}

\newcommand{\Ai}{\text{\normalfont Ai}}
\newtheorem{theo}{Theorem}[section]
\newtheorem{lemma}[theo]{Lemma}
\newtheorem{prop}[theo]{Proposition}
\newtheorem{coro}[theo]{Corollary}
\newtheorem{definition}[theo]{Definition}

\theoremstyle{remark}
\newtheorem{remark}[theo]{Remark}

\newcommand{\OEIS}[1]{\href{http://oeis.org/#1}{OEIS~#1}}

\newtheoremstyle{conjecture}{}{}{\it}{}{\color{purple}\bfseries}{}{ }{}
\theoremstyle{conjecture}

\newcommand{\tdef}[1]{
{\emph{#1}}}
\newcommand{\todyck}{\mathbf{Dyck}}
\newcommand{\dycklabel}{\mathbf{\mathcal{L}}}
\newcommand{\cuniquedecopath}{C-decorated}
\newcommand{\csimpledecopath}{H-decorated}

\newcommand{\ddyck}{\mathcal{D}}

\begin{document}

\author{
Andrew Elvey Price%
\thanks{Andrew Elvey Price was supported by the European Research Council (ERC) in the European Union’s Horizon 2020 research and innovation programme, under the Grant Agreement No.~759702.}
	\and
Wenjie Fang%
\thanks{Wenjie Fang was supported by the Austrian Science Fund (FWF) grant P 27290 and I 2309-N35 as a postdoctoral fellow at TU Graz, during which part of the work was done.}
	\and 
Michael Wallner%
\thanks{Michael Wallner was supported by the Erwin Schr{\"o}dinger Fellowship of the Austrian Science Fund (FWF):~J~4162-N35.}
}

\newcommand{\addorcid}[1]{\protect\includegraphics[height=3mm]{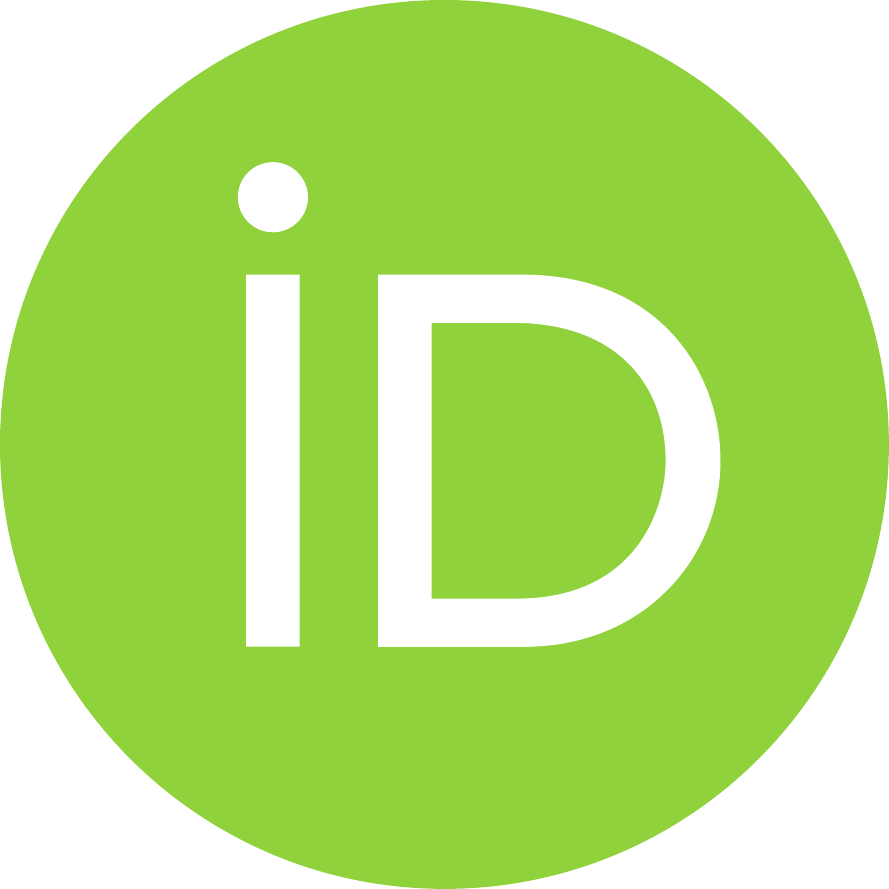} \url{https://orcid.org/#1}}
\newcommand{\Addresses}{{
  \bigskip
  \footnotesize

  A.~Elvey Price, \textsc{Laboratoire Bordelais de Recherche en Informatique, UMR 5800, Universit\'e de Bordeaux, 351 Cours de la Libération, 33405 Talence Cedex, France}\par\nopagebreak
  \textit{E-mail address:} \texttt{andrew.elvey@univ-tours.fr}\par\nopagebreak
\addorcid{0000-0003-3240-6390}

  \medskip

  W.~Fang, \textsc{Laboratoire d'Informatique Gaspard-Monge, UMR 8049, Université Gustave-Eiffel, CNRS, ESIEE Paris, 5 Boulevard Descartes, 77454 Marne-la-Vallée, France}\par\nopagebreak
  \textit{E-mail address:} \texttt{wenjie.fang@u-pem.fr}\par\nopagebreak
  \textit{Website:} \url{http://igm.univ-mlv.fr/~wfang/} \par\nopagebreak
  \addorcid{0000-0001-9148-2807}

  \medskip

  M.~Wallner, \textsc{Laboratoire Bordelais de Recherche en Informatique, UMR 5800, Universit\'e de Bordeaux, 351 Cours de la Libération, 33405 Talence Cedex, France;
	 Institute of Discrete Mathematics and Geometry, TU Wien, Wiedner Hauptra{\ss}e 8--10, 1040 Wien, Austria}\par\nopagebreak
  \textit{E-mail address:} \texttt{michael.wallner@tuwien.ac.at}\par\nopagebreak
  \textit{Website:} \url{https://dmg.tuwien.ac.at/mwallner/} \par\nopagebreak
  \addorcid{0000-0001-8581-449X}

}}


\title{Compacted binary trees admit a stretched exponential}

\maketitle

\let\thefootnote\relax\footnotetext{{\textcopyright}~2020. This manuscript version is made available under the CC-BY-NC-ND 4.0 license \url{http://creativecommons.org/licenses/by-nc-nd/4.0/}}

\renewenvironment{abstract}
{\small
\begin{center}
\bfseries \abstractname\vspace{-.5em}\vspace{0pt}
\end{center}
\list{}{%
	\setlength{\leftmargin}{7mm}
	\setlength{\rightmargin}{\leftmargin}%
}%
\item\relax}
{\endlist}

\begin{abstract}
	A compacted binary tree is a directed acyclic graph encoding a binary tree in which common subtrees are factored and shared, such that they are represented only once.
	We show that the number of compacted binary trees of size $n$ grows asymptotically like
	\begin{align*}
		\Theta\left( n! \, 4^n e^{3a_1n^{1/3}} n^{3/4} \right),
	\end{align*}
	where $a_1\approx-2.338$ is the largest root of the Airy function. Our method involves a new two parameter recurrence which yields an algorithm of quadratic arithmetic complexity for computing the number of compacted trees up to a given size. We use empirical methods to estimate the values of all terms defined by the recurrence, then we prove by induction that these estimates are sufficiently accurate for large $n$ to determine the asymptotic form.
	Our results also lead to new bounds on the number of minimal finite automata recognizing a finite language on a binary alphabet. As a consequence, these also exhibit a stretched exponential.

\end{abstract}

\section{Introduction}
\label{sec:intro}

Compacted binary trees are a special class of directed acyclic graphs that appear as a model for data structures in the compression of XML documents~\cite{bousquet2015xml}. 
Given a rooted binary tree of size $n$, its compacted form can be computed in expected and worst-case time $\LandauO(n)$ with expected compacted size $\Theta(n/ \sqrt{\log n})$~\cite{flss90}. 
Recently, Genitrini, Gittenberger, Kauers, and Wallner  solved the reversed question on the asymptotic number of compacted trees under certain height restrictions~\cite{GenitriniGittenbergerKauersWallner2016}; however the asymptotic number in the unrestricted case remained elusive. They also solved this problem for a simpler class of trees known as relaxed trees under the same height restrictions. 
In this paper we show that the counting sequences $(c_n)_{n \in \N}$ of (unrestricted) compacted binary trees and $(r_n)_{n \in \N}$ of (unrestricted) relaxed binary trees both admit a stretched exponential: 

\begin{theo}
	\label{theo:mainasy}
	The number of compacted and relaxed binary trees satisfy for $n \to \infty$
	\begin{align*}
		c_n &= \Theta\left( n! \, 4^n e^{3a_1n^{1/3}} n^{3/4} \right) &\text{ and } &&		
		r_n &= \Theta\left( n! \, 4^n e^{3a_1n^{1/3}} n \right), 
	\end{align*}
	 where $a_{1}\approx-2.338$ is the largest root of the Airy function $\Ai(x)$ defined as the unique function satisfying $\Ai''(x)=x\Ai(x)$ and $\lim_{n\to \infty} \Ai(x) = 0$.
\end{theo}

We believe that there are constants $\gamma_{c}$ and $\gamma_{r}$ such that 
\[c_{n}\sim \gamma_{c}n!4^{n}e^{3a_1n^{1/3}}n^{3/4}~~~~\text{and}~~~~r_{n}\sim \gamma_{r}n!4^{n}e^{3a_1n^{1/3}}n,\]
however, we have been unable to find the exact values of these constants or even prove their existence. Nevertheless, our empirical analysis yields what we believe to be very accurate estimates for $\gamma_{c}$ and $\gamma_{r}$, namely $\gamma_{c}\approx 173.12670485$ and $\gamma_{r}\approx 166.95208957$.

The presence of a stretched exponential term in a sequence counting combinatorial objects is not common, although there are quite a few precedents. One simple example is that of {\em pushed Dyck paths}, where Dyck paths of maximum height $h$ are given a weight $y^{-h}$ for some $y>1$. In this case McKay and Beaton determined the weighted number $d_{n}$ of paths of length $2n$ up to {\em and including} the constant term to be asymptotically given by
\[d_{n}\sim A y(y-1)(\log y)^{1/3}4^n\exp\left(-C(\log y)^{2/3}n^{1/3}\right)n^{-5/6},\]
where $A=2^{5/3}\pi^{5/6}/\sqrt{3}$ and $C=3(\pi/2)^{2/3}$; see 
\cite{guttmann2015analysis}. For the analogous problem of counting pushed self avoiding walks, Beaton et al.~\cite{BeatonEtal2015Compressed} gave a (non-rigorous) probabilistic argument for the presence of a stretched exponential of the form $e^{-cn^{3/7}}$ for some $c>0$. In each of these cases, a stretched exponential appears as part of a compromise between the large height regime in which most paths occur and the small height regime in which the weight is maximized. We will see that a similar compromise occurs in this paper. 
Another situation in which stretched exponentials have appeared is in cogrowth sequences in groups~\cite{EP2017numerical}, that is, paths on Cayley graphs which start and end at the same point. In particular, Revelle~\cite{revelle2003heat}  showed that in the lamplighter group the number $c_{n}$ of these paths of length $2n$ behaves like
\[c_n \sim C \, 9^n \kappa^{n^{1/3}} n^{1/6}.\]
In the group $\mathbb{Z}\wr\mathbb{Z}$, Pittet and Saloff-Coste showed that the asymptotics of the cogrowth series contains the slightly more complicated term $\kappa^{\sqrt{n\log n}}$ \cite{pittet2002random}. Another example comes from the study of pattern avoiding permutations, where Conway, Guttmann, and Zinn-Justin~\cite{ConwayEtal20151324,ConwayEtal20181324} have given compelling numerical evidence that the number $p_{n}$ of 1324-avoiding permutations of length $n$ behaves like
\[p_{n}\sim B \mu^n \mu_1^{\sqrt{n}} n^g,\]
with $\mu\approx11.600$, $\mu_1 \approx 0.0400$, $g \approx -1.1$. 

As seen by these examples, it is generally quite difficult to prove that a sequence has a stretched exponential in its asymptotics. Part of the difficulty is that a sequence which has a stretched exponential cannot be ``very nice''. In particular, the generating function cannot be algebraic, and can only be $D$-finite if it has an irregular singularity~\cite{flaj09}.

Some explicit examples of $D$-finite generating series with a stretched exponential are known; see e.g.~\cite{Wright1949Coefficients,Wright1933Coefficients,Wright1932Coefficients}. 
In these cases Wright uses a saddle-point method to prove the presence of the stretched exponential. To apply this method, one needs to meticulously check various analytic conditions on the generating function, or to bound related integrals in a delicate way. These tasks can be highly non-trivial and require a precise knowledge of the analytic properties of the generating function. For more detail on how to use the saddle-point method to prove stretched exponentials, and further examples, see \cite[Chapter VIII]{flaj09}. 

In lieu of detailed information on the generating function, we find and analyze  the following recurrence relation
\[r_{n,m}=r_{n,m-1}+(m+1)r_{n-1,m},\]
corresponding to a partial differential equation to which the saddle point method cannot be readily applied.
The number of relaxed trees of size $n$ is then $r_{n,n}$.
We present a method that works directly with a transformed sequence $d_{n,k}$ and the respective recurrence relation. 
We find two explicit sequences $A_{n,k}$ and $B_{n,k}$ with the same asymptotic form, such that
\begin{align}
	\label{eq:boundsintro}
	A_{n,k} \leq d_{n,k} \leq B_{n,k},
\end{align}
for all $k$ and all $n$ large enough. 
The idea is that $A_{n,k}$ and $B_{n,k}$ satisfy the recurrence of $d_{n,k}$ with the equalities replaced by inequalities, allowing us to prove~\eqref{eq:boundsintro} by induction. 
In order to find appropriate sequences $A_{n,k}$ and $B_{n,k}$, we start by performing a heuristic analysis to conjecture the asymptotic shape of $d_{n,k}$ for large $n$. We then prove that the required recursive inequalities hold for sufficiently large $n$ using adapted Newton polygons.

The inductive step in the method described above requires that all coefficients in the recurrence be positive. This occurs in the case of relaxed binary trees but not for compacted binary trees. In the latter case, we construct a sandwiching pair of sequences, each determined by a recurrence with positive coefficients, to which our method applies. 

As an application, we use our results on relaxed and compacted trees to give new asymptotic upper and lower bounds for the number of minimal deterministic finite automata with $n$ states recognizing a finite language on a binary alphabet. These automata are studied in the context of the complexity of regular languages; see~\cite{DomaratzkiKismaShallit2002DFA,liskovets2006exact, domaratzki2006enumeration}. To our knowledge no upper or lower bounds capturing even the exponential term had been proven for this problem. Our bounds are much more accurate, only differing by a polynomial factor, and thereby proving the presence of a stretched exponential term.

As a further extension of our method, some preliminary results show that our approach can be generalized to a $k$-ary version of compacted trees, which in turn settles the enumeration of minimal finite automata recognizing finite languages for an arbitrary alphabet. A follow-up paper in this direction is underway.

In its simplest form, our method applies to two parameter linear recurrences with positive coefficients which may depend on both parameters. We expect, however, that our method could be adapted to handle a much wider range of recurrence relations, potentially involving more than two parameters, negative coefficients and perhaps even some non-linear recurrences. Indeed, we have already seen that it can be adapted to at least one case involving negative coefficients, namely that of counting compacted binary trees.

\paragraph{Plan of the article.}
In Section~\ref{sec:paths} we introduce compacted binary trees and the related relaxed binary trees, and then derive a bijection to Dyck paths with weights on their horizontal steps.
In Section~\ref{sec:heuristic} we show a heuristic method of how to conjecture the asymptotics and in particular the appearance of a stretched exponential term.
Building on these heuristics, we prove exponentially and polynomially tight bounds for the recurrence of relaxed binary trees in Section~\ref{sec:induction} and of compacted binary trees in Section~\ref{sec:compacted}.
In Section~\ref{sec:automata} we show how our results lead to new bounds on minimal acyclic automata on a binary alphabet.

\section{A two-parameter recurrence relation}
\label{sec:paths}

Originally, compacted binary trees arose in a compression procedure in~\cite{flss90} which computes the number of unique fringe subtrees. 
Relaxed binary trees are then defined by relaxing the uniqueness conditions on compacted binary trees. 
As we will not need this algorithmic point of view, we directly give the following definition adapted from~\cite[Definition~3.1 and Proposition~4.3]{GenitriniGittenbergerKauersWallner2016}.

Before we define compacted and relaxed binary trees, let us recall some basic definitions.
A \emph{rooted binary tree} is a plane directed connected graph with a distinguished node called the root, in which all nodes have out-degree either $0$ or $2$ and all nodes other than the root have in-degree $1$, while the root has in-degree $0$. For each vertex with out-degree $2$, the out-going edges are distinguished as a left edge and a right edge. Nodes with out-degree $0$ are called \emph{leaves}, and nodes with out-degree $2$ are called \emph{internal nodes}.
All trees in this paper will be rooted and we omit this term in the future.

\begin{definition}[Relaxed binary tree] \label{def:relaxedtree}
A \emph{relaxed binary tree} (or simply \emph{relaxed tree}) of size~$n$ is a directed acyclic graph obtained from a binary tree with $n$ internal nodes, called its \emph{spine}, by keeping the left-most leaf and turning other leaves into pointers, with each one pointing to a node (internal ones or the left-most leaf) preceding it in postorder.
\end{definition}

The counting sequence $(\rt_n)_{n \in \N}$ of relaxed binary trees of size $n$ starts as follows:
\begin{align*}
	\left( \rt_n \right)_{n \in \N} &= 
		\left( 
			1, 1, 3, 16, 127, 1363, 18628, 311250, 6173791, 142190703, 
		\ldots \right).
\end{align*}
It corresponds to \OEIS{A082161} in the On-line Encyclopedia of Integer Sequences.\footnote{\url{https://oeis.org}} 
There, it first appeared as the counting sequence of the number of deterministic, completely defined, initially connected,
acyclic automata with $2$ inputs and $n$ transient, unlabeled states
and a unique absorbing state, yet without specified final states.
This is a direct rephrasing of Definition~\ref{def:relaxedtree} in the language of automata theory; for more details see Section~\ref{sec:automata}.
Liskovets~\cite{liskovets2006exact} provided (probably) the first recurrence relations ($C_2(n)$ used for $r_n$) and 
later Callan~\cite{Callan2008Determinant} showed that they are counted by determinants of Stirling cycle numbers.
However, the asymptotics remained an open problem, which we will solve in the present paper.

%
\begin{figure}[ht!]
	\centering
	\includegraphics[width=0.8\textwidth]{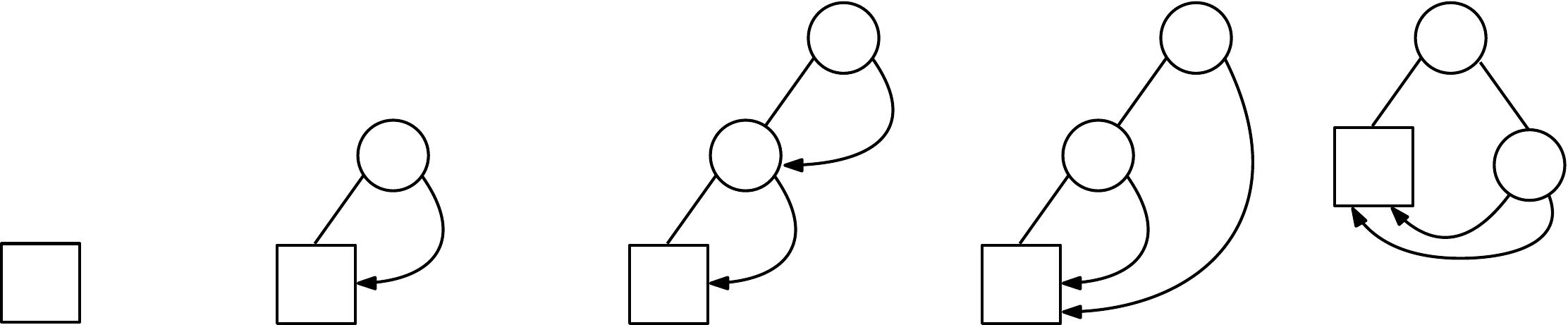}%
	\caption{All relaxed (and also compacted) binary trees of size $0,1,2$, where internal nodes are shown by circles and the unique leaf is drawn as a square. }
	\label{fig:compacted_trees_n123}
\end{figure}

Using the class of relaxed trees, it is then easy to define the set of compacted trees by requiring the uniqueness of subtrees.

\begin{definition}[Compacted binary tree]\label{def:compactedtree}
Given a relaxed tree, to each node $u$ we can associate a binary tree $B(u)$. We proceed by postorder. If $u$ is the left-most leaf, we define $B(u) = u$. Otherwise, $u$ has two children $v, w$, then $B(u)$ is the binary tree with $B(v)$ and $B(w)$ as left and right sub-trees, respectively. 
A \emph{compacted binary tree}, or simply \emph{compacted tree} of size~$n$ is a relaxed tree with $B(u) \neq B(v)$ (i.e., $B(u)$ not isomorphic to $B(v)$) for all pairs of distinct nodes $u,v$.
\end{definition}

Figure~\ref{fig:compacted_trees_n123} shows all relaxed (and compacted) trees of size $n=0,1,2$ and 
Figure~\ref{fig:compacted3invalid} gives the smallest relaxed tree that is not a compacted tree.
The counting sequence $(\ct_n)_{n \in \N}$ of compacted binary trees of size $n$ is \OEIS{A254789} and starts as follows:
\begin{align*}
	\left( \ct_n \right)_{n \in \N} &= 
		\left( 
			1, 1, 3, 15, 111, 1119, 14487, 230943, 4395855, 97608831, 
		\ldots \right).	
\end{align*}


\begin{figure}[ht!]
	\centering
	\includegraphics[width=0.19\textwidth]{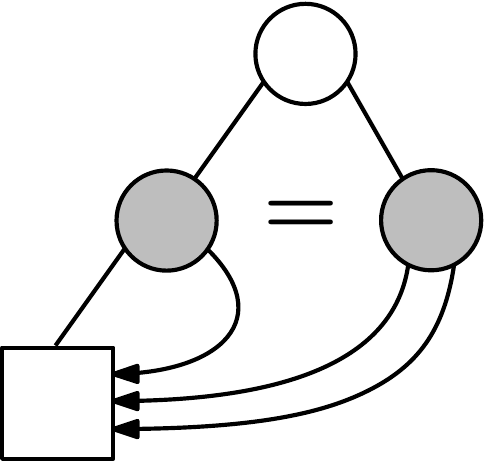}%
	\qquad \qquad \qquad \qquad 
	\includegraphics[width=0.19\textwidth]{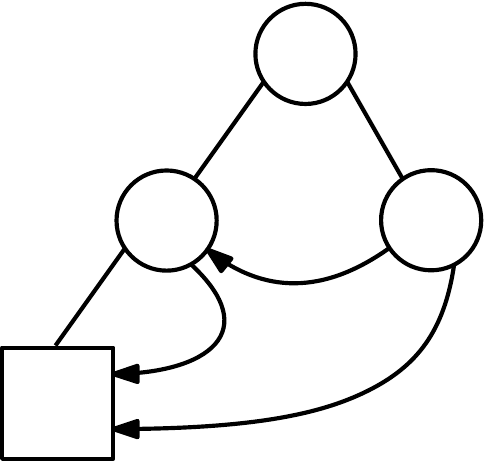}%
	\caption{(Left) The smallest relaxed binary tree that is not a compacted binary tree, as the two gray subtrees correspond to the same (classical) binary tree. (Right) A valid compacted binary tree of size $3$ with the same spine. 
}
	\label{fig:compacted3invalid}
\end{figure}

In~\cite[Theorem~5.1 and Corollary~5.4]{GenitriniGittenbergerKauersWallner2016} the so-far most efficient recurrences are given for the number of compacted and relaxed binary trees, respectively. 
Computing the first $n$ terms using these requires $\LandauO(n^3)$ arithmetic operations. 
In this section we give an alternative recurrence with only one auxiliary parameter (instead of two) other than the size $n$, which leads to an algorithm of arithmetic complexity $\LandauO(n^2)$ to compute the first $n$ terms of the sequence. 
The construction is motivated by the recent bijection~\cite{Wallner2018R1}.

As a corollary of our main result Theorem~\ref{theo:mainasy}, we directly get an estimate of the asymptotic proportion of compacted trees among relaxed trees:
\begin{align*}
	\frac{c_n}{r_n} = \Theta(n^{-1/4}).
\end{align*}
An analogous result for compacted and relaxed trees of bounded right height was shown in~\cite[Corollary~3.5]{GenitriniGittenbergerKauersWallner2016}. The right height is the maximal number of right edges to internal nodes on a path in the spine from the root to a leaf. Let $c_{k,n}$ (resp.~$r_{k,n}$) be the number of compacted (resp.~relaxed) trees of right height at most~$k$.
Then, 
for fixed $k$, 
\begin{align*}
	\frac{c_{k,n}}{r_{k,n}} \sim 
		\lambda_k n^{- \frac{1}{k+3} - \left(\frac{1}{4} - \frac{1}{k+3}\right)\frac{1}{\cos^2\left(\frac{\pi}{k+3} \right)}}
		= \Landauo\left(n^{-1/4} \right),
\end{align*}
for a constant $\lambda_k$ independent of $n$. As $k \to \infty$, we see that the exponent of $n$ approaches~$-1/4$. It is thus not surprising that the exponent in the unbounded case is also~$-1/4$.

\subsection{Relaxed binary trees and horizontally decorated paths}

For the subsequent construction, we need the following type of lattice paths. 

\begin{definition}
	\label{def:hordecpath}
	A \emph{horizontally decorated path} $P$ is a lattice path starting from $(0,0)$ with steps $H = (1,0)$ and $V = (0,1)$ confined to the region $0 \leq y \leq x$, where each horizontal step $H$ is decorated by a number in $\{1,\ldots,k+1\}$ with $k$ its $y$-coordinate. If $P$ ends at $(n,n)$, we call it a \emph{horizontally decorated Dyck path}.
\end{definition}

We denote by $\ddyck_n$ the set of horizontally decorated Dyck paths of length $2n$.

\begin{remark}
	Horizontally decorated Dyck paths can also be interpreted as classical Dyck paths, where below every horizontal step a box given by a unit square between the horizontal step and the line $y=-1$ is marked, see Figure~\ref{fig:bij-dyck}. 
	This gives an interpretation connecting these paths with the heights of Dyck paths, which we will exploit later. 
	Independently, Callan gave in~\cite{Callan2008Determinant} a more general bijection in which he called the paths \emph{column-marked subdiagonal paths}, and Bassino and Nicaud studied in~\cite{BassinoNicaud2007Automata} a variation when counting some automata, where the paths stay \emph{above} the diagonal, which they called \emph{$k$-Dyck boxed diagrams}.
\end{remark}

\begin{theo}
	\label{theo:relaxedbij}
	There exists a bijection $\todyck$ between relaxed binary trees of size $n$ and the set $\ddyck_n$ of horizontally decorated Dyck paths of length $2n$. 
\end{theo} 

\begin{proof}
	Let $C$ be a relaxed binary tree of size $n$, and $C_*$ its spine. For convenience, we identify the internal nodes in $C$ and $C_*$, and pointers in $C$ with leaves (not the left-most one) in $C_*$.
	
	We now give a recursive procedure transforming $C$ into a horizontally decorated Dyck path $P$. 
	First, we take $C_*$ and label its internal nodes \emph{and the left-most leaf} in postorder from $1$ to $n+1$.
	Next, we define the following function $\mathbf{Path}$ that transforms $C_*$ into a lattice path in $H$ and $V$. Given a binary tree $T$, it either consists of two sub-trees $(T_1, T_2)$, or it is a leaf $\varepsilon$. We thus define $\mathbf{Path}$ recursively by
	\[
	\mathbf{Path}((T_1, T_2)) = \mathbf{Path}(T_1)\mathbf{Path}(T_2)V, \quad\quad \mathbf{Path}(\varepsilon) = H.
	\]
	It is clear that $\mathbf{Path}(C_*)$ starts with $H$ for the left-most leaf. Let $P_0$ be $\mathbf{Path}(C_*)$ with its starting $H$ removed. 	
	Note that $\mathbf{Path}$ performs a postorder traversal on $C_*$ where leaves are matched with $H$ and internal nodes with $V$. Then, $\mathbf{Path}(C_*)$ ends at $(n+1,n)$ and stays always strictly below $y=x$ because every binary (sub-)tree has one more leaf than internal nodes, and each initial segment of $\mathbf{Path}(C_*)$ corresponds to a collection of subtrees of $C_*$. Hence, $P_0$ is a Dyck path. 
	Observe that the $i$-th step $V$ in $P_0$ corresponds to the $(i+1)$-st node in postorder, as the left-most leaf is labeled~$1$.
	Finally, for each step $H$ in $P_0$, we label it by the label of the internal node (or the left-most leaf) to which its corresponding leaf in $C_*$ points in $C$. We thus obtain a Dyck path~$P$ with labels on the horizontal steps, and we define $\todyck(C_*) = P$.
	
	We have seen that the Dyck path $P_0$ is in bijection with the spine $C_*$. To see that the labeling condition on horizontally decorated Dyck paths is equivalent to the condition on relaxed binary trees, we take a pointer $p$ pointing to a node $u$ with label $\ell$ that corresponds to a step $H$ with a certain coordinate $k$. By construction of the Dyck path, $p$ comes after $u$ in postorder if and only if the step $H$ from $p$ comes after the step $V$ from $u$, which is equivalent to $\ell \leq k+1$, as the node with label $1$ is the left-most leaf and is not recorded as a step $H$. We thus have the equivalence of the two conditions, so $\todyck$ is indeed a bijection as claimed.
\end{proof}

\begin{figure}[thbp]
\centering
\includegraphics[width=\textwidth, page=1]{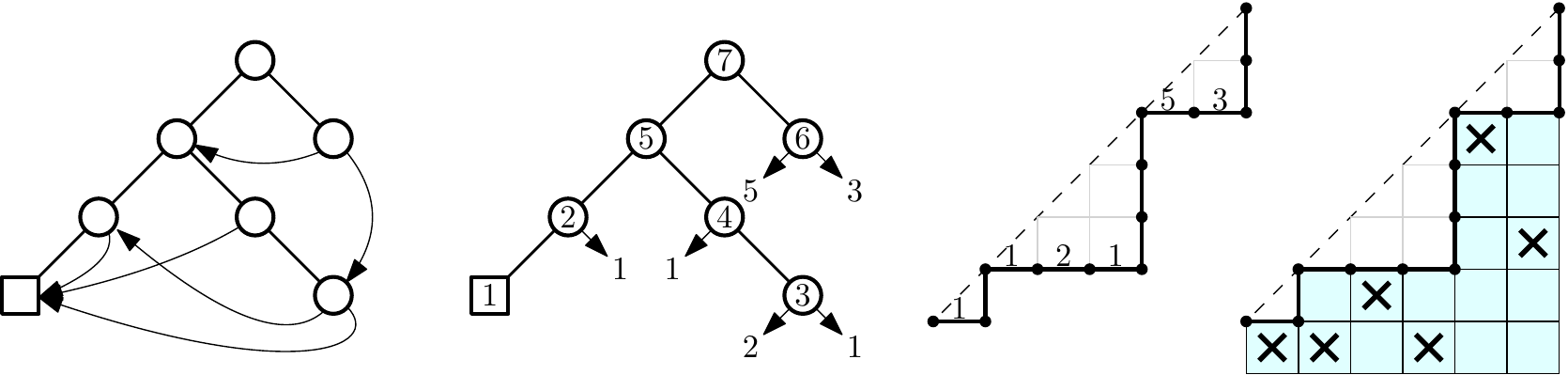}
\caption{Example of the bijection $\todyck$ between relaxed trees and horizontally decorated Dyck paths. It transforms internal nodes into vertical steps and pointers into horizontal steps.}
\label{fig:bij-dyck}
\end{figure}

The following result gives the claimed algorithm with quadratic arithmetic complexity to count such paths, which can also be used as a precomputation step of an algorithm that randomly generates these paths using a linear number of arithmetic operations for each path. These algorithms are also applicable to relaxed binary trees via the bijection $\todyck$.

\begin{prop}
	\label{prop:recrelaxed}
	Let $r_{n,m}$ be the number of horizontally decorated paths ending at $(n,m)$. Then,
	\begin{align*}
		r_{n,m} &= r_{n,m-1} + (m+1) r_{n-1,m}, & \text{ for } &n,m \geq 1 \text{ and } n\geq m,\\
		r_{n,m} &= 0, & \text{ for } & n < m, \\
		r_{n,0} &= 1, & \text{ for } & n \geq 0.
	\end{align*}
	The number of relaxed binary trees of size $n$ is equal to $r_{n,n}$.
\end{prop}

\begin{proof}
	Let us start with the boundary conditions. First of all, no such path is allowed to cross the diagonal $y=x$, thus $r_{n,m}=0$ for $n<m$. Second, the paths consisting only of horizontal steps stay at altitude $0$ and admit therefore just one possible label for each step, i.e., $r_{n,0}=1$ for $n \geq 0$. 
	
	For the recursion we consider how a path can jump to $(n,m)$. It either uses a step~$V$ from $(n,m-1)$ or it uses a step~$H$ from $(n-1,m)$. In the second case, there are $m+1$ possible decorations as the path is currently at altitude $m$.
\end{proof}

\begin{remark}[Compacted trees of bounded right height]
	This restriction naturally translates relaxed binary trees of right height at most $k$ from~\cite{GenitriniGittenbergerKauersWallner2016} into horizontally decorated Dyck paths of height at most $k+1$, where height is the maximal normal distance rescaled by $\sqrt{2}$ from a lattice point on the path to the diagonal. 
	In other words, these paths are constrained to remain between the diagonal and a line translated to the right parallel to the diagonal by $k+1$ unit steps.
\end{remark}

\subsection{Compacted binary trees}

Given a relaxed tree $C$, an internal node $u$ is called a \emph{cherry} if its children in the spine are both leaves and none of them is the left-most one. According to the discussion at the end of Section~4 in \cite{GenitriniGittenbergerKauersWallner2016}, the only obstacle for a relaxed tree to be a compacted tree is a cherry with badly chosen pointers. For the convenience of the reader, we now recall and formalize this observation in the following proposition.

\begin{prop} \label{prop:cherry}
A relaxed tree $C$ is a compacted tree if and only if there are no two nodes $u\neq v$ in $C$ which share the same left child $u_{\ell}$ and the same right child $u_{r}$. Moreover, if $C$ is not a compacted tree, such a pair exists where $v$ is a cherry and $u$ precedes $v$ in postorder.
\end{prop}
\begin{proof}
The ``only if'' part follows directly from Definition~\ref{def:compactedtree}. We now focus on the ``if'' part. Suppose that $C$ is not a compacted tree, which means there is at least a pair of internal nodes $u, v$ such that $u$ precedes $v$ and $B(u) = B(v)$, with $B(u)$ defined in Definition~\ref{def:compactedtree}. 
Now we want to show that there is one such pair with $v$ being a cherry. 
We take such a pair $(u,v)$. 
If $v$ is a cherry, the claim holds. 
Otherwise, without loss of generality, we suppose that the left child $v'$ of $v$ is not a leaf. Let $u_\ell$ be the left child of $u$. If $u_\ell$ is an internal node, we take $u' = u_\ell$. Otherwise, we take $u'$ to be the internal node pointed to by $u_\ell$. By definition, we have $B(u') = B(v')$, and clearly $u'$ precedes $v'$ in postorder. We thus obtain a new pair with the same conditions but of greater depth in the spine. However, since the spine has finite depth, this process cannot continue forever. As it only stops when $v$ is a cherry, we have the existence of such a pair $(u,v)$ with $v$ a cherry.
\end{proof}

The restriction  described in Proposition \ref{prop:cherry} has an analogue in the class of horizontally decorated paths:
We label every step $V$ with its final altitude plus one, which corresponds to its row number in the interpretation with marked boxes, and which also corresponds to the traversal/process order in postorder of its internal node in the relaxed tree; compare Figure~\ref{fig:bij-dyck}. 
Recall that each step $H$ is already labeled. 
For any step $S$, let $\dycklabel(S)$ be its label.
We associate to every step $V$ a pair of integers $(v_1,v_2)$, which correspond to the labels of its left and right children.
First, let $S'$ be the step before $V$ and set $v_2 = \dycklabel(S')$.
Next, draw a line from the ending point of $V$ in the southwest direction parallel to the diagonal, and stop upon touching the path again. Let $S''$ be the last step before $V$ that ends on this line (if there is no such step, set $v_{1}=1$). Then set $v_{1} = \dycklabel(S'')$.

\begin{definition}
	\label{def:hordecpathb}
	A \emph{\cuniquedecopath{} path} $P$ is a horizontally decorated path where the decorations $h_1$ and $h_2$ of each pattern of consecutive steps $HHV$ fulfill $(h_1,h_2) \neq (v_1,v_2)$ for all preceding steps~$V$. 
\end{definition}

\begin{prop}
	The map $\todyck$ bijectively sends the set of compacted trees of size~$n$ to the set of \cuniquedecopath{} Dyck paths of length~$2n$.
\end{prop}

\begin{proof} 
Recall from Theorem~\ref{theo:relaxedbij} that the map $\todyck$ is a bijection sending relaxed trees of size $n$ to the set of horizontally decorated Dyck paths of size $2n$.	C-decorated paths are defined precisely so that their corresponding relaxed trees satisfy the condition of Proposition~\ref{prop:cherry}. Therefore, $\todyck$ forms a bijection between \cuniquedecopath{} paths and compacted trees.
\end{proof}

The key observation for the counting result is that exactly one pair of labels $(h_{1},h_{2})$ is avoided for each preceding step~$V$ of a consecutive pattern $HHV$.
Applying this classification to the previous result we get a similar quadratic-time recurrence for compacted binary trees.

\begin{prop} \label{prop:reccompacted}
	Let $c_{n,m}$ be the number of \cuniquedecopath{} paths ending at $(n,m)$. Then,
	\vspace{-1mm}
	\begin{align*}
		c_{n,m} &= c_{n,m-1} + (m+1) c_{n-1,m} - (m-1) c_{n-2,m-1}, & \text{ for } &n\geq m \geq 1,\\
		c_{n,m} &= 0, & \text{ for } & n < m, \\
		c_{n,0} &= 1, & \text{ for } & n \geq 0.
	\end{align*}
	\vspace{-1mm}
	The number of compacted binary trees of size $n$ is equal to $c_{n,n}$.
\end{prop}
\begin{proof}
In the first case, the term $(m+1)c_{n-1,m}$ counts the paths ending with a $H$-step while $c_{n,m-1}-(m-1)c_{n-2,m-1}$ counts the paths ending with a $V$-step. The term $-(m-1)c_{n-2,m-1}$ occurs because, for each  \cuniquedecopath{} path ending at $(n-2,m-1)$, there are exactly $m-1$ paths formed by adding an additional $HHV$ that are not \cuniquedecopath{} paths.
\end{proof}

Note that one might also count the following simpler class which is in bijection with \cuniquedecopath{} paths, albeit without a natural bijection. 

\begin{definition}
	\label{def:hordecpathc}
	A \emph{\csimpledecopath{} path} $P$ is a horizontally decorated path where the decorations $h_1$ and $h_2$ of each pattern of consecutive steps $HHV$ fulfill $h_1 \neq h_2$ except for $h_1=h_2=1$.
\end{definition}

In terms of marked boxes, this constraint translates to the fact that, below the horizontal steps in each consecutive pattern $HHV$, the marks must be in different rows except possibly for the lowest one.

\section{Heuristic analysis} \label{sec:heuristic}

In this section, we will explain briefly some heuristics and an \textit{ansatz} that we will apply later to get the asymptotic behavior of $r_n$ and $c_n$. These heuristics are closely related to the asymptotic behavior of Dyck paths and the Airy function.

\subsection{An intuitive explanation of the stretched exponential} \label{sec:heuristic-dyck}

We can consider $r_n$ as a weighted sum of Dyck paths, where each Dyck path $P$ has a weight $w(P)$ that is the number of horizontally decorated Dyck paths that it gives rise to. There is thus a balance of the number of total paths and their weights for the weighted sum $r_{n,n}$. On the one hand, most paths have an (average) height of $\LandauO(\sqrt{n})$ (i.e., mean distance to the diagonal). On the other hand, their weight is maximal if their height is $\LandauO(1)$, i.e., they are close to the diagonal. In other words, typical Dyck paths are numerous but with small weight, and Dyck paths atypically close to the diagonal are few but with enormous weight. The asymptotic behavior of the weighted sum of Dyck paths that we consider should be a result of a compromise between these two forces. We will now make this more explicit by analyzing Dyck paths with height approximately $n^{\alpha}$ for some $\alpha\in(0,1/2)$.

Given a Dyck path $P$ with steps $H=(1,0)$ and $V=(0,1)$ as in Definition~\ref{def:hordecpath}, let $m_i$ be the $y$-coordinate of  the $i$-th step $H$.  
The number of Dyck paths with $m_i$ bounded uniformly satisfy the following property. 

\begin{prop}[{\cite[Theorem~3.3]{kousha2012asymptotic}}]
  For a Dyck path $P$ of length $2n$ chosen uniformly at random, let $m_i$ be the $y$-coordinate of the $i$-th step $H$. For $\alpha < 1/2$, we have
  \[
  \log \PR\left(\max_{1 \leq i \leq n} (i - m_i) < n^{\alpha} \right) \sim -\pi^2 n^{1 - 2\alpha}.
  \]
\end{prop}

Let $w(P)$ the number of horizontally decorated Dyck paths whose unlabeled version is the Dyck path $P$. For a randomly chosen Dyck path $P$ of length $2n$ with $i-m_i$ bounded uniformly by $n^\alpha$, we heuristically expect most values of $i-m_i$ to be of the order $\Theta(n^\alpha)$, with $i$ of order $\Theta(n)$. 
This leads to the following approximation:
\begin{align*}
\log \frac{w(P)}{n!} = \sum_{1 \leq i \leq n} \log\left( \frac{m_i+1}{i} \right) = \sum_{1 \leq i \leq n} \log\left( 1 - \frac{i - m_i-1}{i} \right) \approx c n \cdot \left( - \frac{n^\alpha}{n} \right) = - c n^\alpha.
\end{align*}
Here, $c>0$ is some constant depending on $\alpha$. This approximation is only heuristically justified and very hard to prove. The contribution of Dyck paths with $i-m_i$ uniformly bounded by $n^\alpha$ should thus roughly be $n! 4^n \exp(-(1+o(1))c' n^{p(\alpha)})$, with $p(\alpha) = \min(\alpha, 1 - 2\alpha)$ and $c'>0$ a constant depending on $\alpha$. Here, $4^n$ comes from the growth constant of Dyck paths. The function $p(\alpha)$ is minimal at $\alpha = 1/3$, which maximizes the contribution, leading to the following heuristic guess that the number of relaxed binary trees $r_n$ should satisfy
\[
\log \frac{r_n}{n! 4^n} \underset{n \to \infty}{\sim} - a n^{1/3}, 
\]
for some constant $a>0$. Furthermore, we anticipate that the main contribution should come from horizontally decorated Dyck paths with $i - m_i$ mostly of order $\Theta(n^{1/3})$. Since most such $i$'s should be of order $\Theta(n)$, we can even state the condition above as $x - y = \Theta(y^{1/3})$ for most endpoints $(x,y)$ of horizontal steps. This heuristic is the starting point of our analysis.

\subsection{Weighted Dyck meanders}
\label{sec:meander}
\newcommand{\rs}{d}
\newcommand{\rsr}{\tilde{d}}

The heuristics of the previous section suggest that the mean distance to the diagonal will play an important role. 
Therefore, we propose another model of lattice paths emphasizing this distance. 
A \tdef{Dyck meander} (or simply a \tdef{meander}) $M$ is a lattice path consisting of up steps $U = (1,1)$ and down steps $D = (1,-1)$ while never falling below $y=0$. It is clear that Dyck paths of length $2n$ are in bijection with Dyck meanders of length $2n$ ending on $y=0$ with the transcription $H \to U, V \to D$. This bijection can also be viewed geometrically as the linear transformation $x' = x + y, y' = x - y$. This transformation will simplify the following analysis. We can consider Dyck meanders as initial segments of Dyck paths. 

Furthermore, we have seen that a rescaling by $n!$ seems practical.
So we consider the following weight on steps $U$ in a meander $M$. If $U$ starts from $(a,b)$, then its weight is $(a-b+2)/(a+b+2)$, and the weight of $M$ is the product of the weights of its steps $U$. Let $\rs_{n,m}$ denote the weighted sum of meanders ending at $(n,m)$. We get the following recurrence for $\rs_{n,m}$.

\begin{prop} \label{prop:rec-meander-relaxed}
The weighted sum $\rs_{n,m}$ defined above for meanders ending at $(n,m)$ satisfies the recurrence
\begin{align}
	\label{eq:relaxedrecsimp}
	\left\{
		\begin{array}{rlrl}
		\rs_{n,m} &= \frac{n-m+2}{n+m} \rs_{n-1,m-1}+ \rs_{n-1,m+1}, & \text{ for } &n> 0, m\geq0,\\
		\rs_{0,m} &= 0, & \text{ for } & m>0, \\
		\rs_{n,-1} &= 0, & \text{ for } & n \geq 0, \\
		\rs_{0,0} &= 1.&&
		\end{array}
	\right.
\end{align}
\end{prop}
\begin{proof}
We concentrate on the first case, as the boundary cases follow directly from the definition of meanders. Given a meander ending at $(n,m)$ with $n > 0$, the last step may be an up step or a down step. The contribution of the former case is $\frac{n-m+2}{n+m} \rs_{n-1,m-1}$, with the weight of the last up step taken into account. The contribution of the latter case is simply $\rs_{n-1,m+1}$. We thus get the claimed recurrence.
\end{proof}

\begin{coro} \label{coro:meander-in-dyck}
For integers $m, n$ of the same parity, we have
\[
\rs_{n,m}=\frac{1}{((n+m)/2)!}r_{(n+m)/2,(n-m)/2}.
\]
When $m,n$ are not of the same parity, we have $\rs_{n,m} = 0$.

In particular, the number of relaxed trees of size $n$ is given by $n! \rs_{2n,0}$.
\end{coro}
\begin{proof}
It is clear that meanders can only end on points $(n,m)$ for $n,m$ of the same parity. In this case, it suffices to compare Proposition~\ref{prop:recrelaxed} with Proposition~\ref{prop:rec-meander-relaxed} under the proposed equality.
\end{proof}

For some simple cases of $\rs_{n,m}$, elementary computations show that $\rs_{n,m} = 0$ for $m>n$, $\rs_{n,n} = \frac{1}{n!}$, $\rs_{n,n-2} = \frac{2^{n-1}-1}{(n-1)!}$ and $\rs_{ n,n-4} = \frac{7 \cdot 3^{n-3} - 2^n + 1}{2(n-2)!}$.

\subsection{Analytic approximation of weighted Dyck meanders} \label{sec:heuristic-analytic}

The heuristic in Section~\ref{sec:heuristic-dyck} suggests that the main weight of $\rs_{n,m}$ comes from the region $m = \Theta(n^{1/3})$. It thus suggests an approximation of $\rs_{n,m}$ of the form
\begin{equation} \label{eq:ansatz}
  \rs_{n,m}\sim f(n^{-1/3}(m+1))h(n),
\end{equation}
for some functions $f$ and $h$, where we expect $h(n)\approx 2^{n}\rho^{n^{1/3}}$ for some $\rho$. The idea is that $h(n)$ describes how the total weight for a fixed $n$ grows, and $f(\kappa)$ describes the rescaled weight distribution in the main region $m = \Theta(n^{1/3})$.

Let $s(n)$ be the ratio $\frac{h(n)}{h(n-1)}$. Suppose that $m=\kappa n^{1/3}-1$, the recurrence relation becomes
\begin{align} \label{eq:rec-equation}
  f(\kappa) s(n) = \frac{n-\kappa n^{1/3}+3}{n+\kappa n^{1/3}-1} f\left((n-1)^{-1/3}(\kappa n^{1/3}-1)\right) + f\left((n-1)^{-1/3}(\kappa n^{1/3}+1)\right).
\end{align}

Now, since we expect $h(n) \approx 2^n \rho^{n^{1/3}}$, we postulate that the ratio $s(n)$ behaves like
\begin{equation} \label{eq:ratio-sn}
  s(n) = 2+cn^{-2/3}+O(n^{-1}),
\end{equation}
and that $f(\kappa)$ is analytic. Using these assumptions, we can expand \eqref{eq:rec-equation} as a Puiseux series in $1/n$. Moving all terms to the right-hand side yields
\[0=\left((c+2\kappa)f(\kappa)-f''(\kappa)\right)n^{-2/3}+O(n^{-1}).\]
Solving the differential equation $(c+2\kappa)f(\kappa)-f''(\kappa)=0$ under the condition $f(\kappa) \to 0$ when $\kappa \to \infty$ yields the unique solution (up to multiplication by a constant)
\begin{equation}
  \label{eq:analytic-approx}
  f(\kappa)=b\Ai\left(\frac{c+2\kappa}{2^{2/3}}\right).
\end{equation}
The condition on the behavior of $f(\kappa)$ near $\infty$ is motivated by the experimental observation that $\rs_{n,m}$ is quickly decaying for $m$ close to $n$.
We also insist that $f(0)=0$ as $\rs_{n,-1}=0$, which implies that $c=2^{2/3}a_{1}$ where $a_{1}\approx-2.338$ is the first root of the Airy function $\Ai(x)$, i.e.~the largest one as all roots are on the real negative axis; see~\cite[p.~450]{AbramowitzStegun1964}.
Now, using this conjectural value of $c$, it follows that (ignoring polynomial terms)
\[h(n)\approx 2^n\exp\left(3a_1 (n/2)^{1/3}\right).\]
This suggests that the number of relaxed trees $r_{n}=n!\rs_{2n,0}$ behaves like
\[r_{n}\approx n!4^{n}\exp\left(3a_1 n^{1/3}\right),\]
which is compatible with what we want to prove.

We observe that \eqref{eq:rec-equation} can be expanded into a Puiseux series of $n^{1/3}$ by taking appropriate series expansions of $f(\kappa)$ and $s(n)$. Therefore, to refine the analysis above, it is natural to look at the expansion of $s(n)$ in \eqref{eq:ratio-sn} to more subdominant terms, and to postulate a more refined \textit{ansatz} of $d_{n,m}$ than \eqref{eq:ansatz}, probably as a series in $n^{1/3}$. Indeed, if we take
\[\rs_{n,m} \sim \left(f(n^{-1/3}(m+1)) + n^{-1/3}g(n^{-1/3}(m+1)) \right) h(n)\]
and
\[s(n) = 2 + c n^{-2/3} + d n^{-1} + O(n^{-4/3}),\]
then using the same method we can reach the polynomial part of the asymptotic behavior of $r_n$ as
\[r_{n} \approx n!4^{n} \exp\left(3a_1 n^{1/3}\right)n.\]
In general, we can postulate
\[\rs_{n,m} \approx h(n) \sum_{j=0}^{k} f_{j}(n^{-1/3}(m+1)) n^{-j/3},\]
and
\[s(n) = 2 + \gamma_{2} n^{-2/3} + \gamma_{3} n^{-1} + \ldots + \gamma_{k} n^{-k/3} + o(n^{-k/3}).\]
The proof of our main result on relaxed binary trees is based on choosing the cutoff appropriately, and using perturbations of that truncation to bound $r_{n}$.

\subsection{Discussion on the constants}
One of the first steps in our method 
involves taking ratios $h(n)/h(n-1)$ (or equivalently $r_{n}/r_{n-1}$) of successive terms. From the leading asymptotic behavior of these ratios we can deduce the exact asymptotic form up to the constant term. Unfortunately, however, this method makes it impossible to exactly determine the constant term $\gamma_{r}$. In this section we give estimates of the constant terms: we believe that there are constants $\gamma_{r}\approx166.95208957$ and $\gamma_{c}\approx173.12670485$ such that 
\[c_{n}\sim \gamma_{c}n!4^{n}e^{3a_1n^{1/3}}n^{3/4}~~~~\text{and}~~~~r_{n}\sim \gamma_{r}n!4^{n}e^{3a_1n^{1/3}}n.\]

\newcommand{\picwidthconst}{165pt}
\begin{figure}[htb!]
	\begin{subfigure}{0.48\textwidth}
		\begin{picture}(190,190)
		\put(18,15){\includegraphics[width=\picwidthconst]{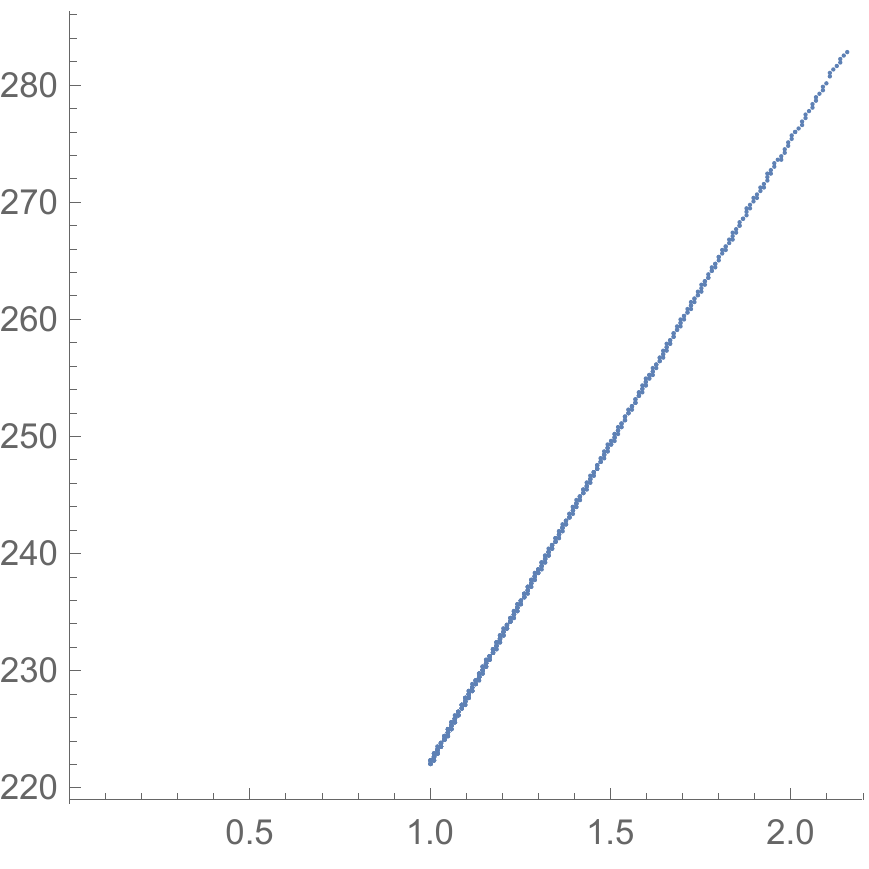}}
		\put(0,100){$u_{n}$}
		\put(100,0){$10n^{-1/3}$}
		\end{picture}
		
		\subcaption{Plot of $u_{n}$ vs.~$10n^{-1/3}$ for approximating the constant term $\gamma_{r}$ of relaxed trees. \\~}
		\label{fig:relaxed_constant_approx_bad}
	\end{subfigure}\hfill
	\begin{subfigure}{0.48\textwidth}
		\begin{picture}(190,190)
		\put(18,15){\includegraphics[width=\picwidthconst]{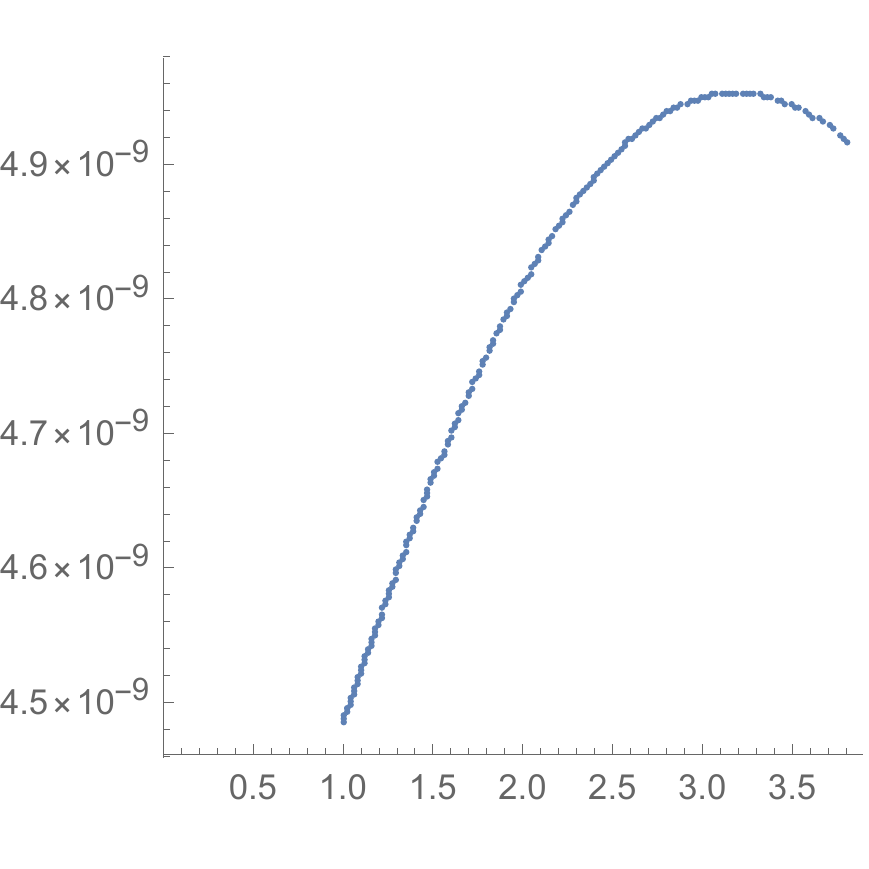}}
		\put(0,100){$\hat{v}_{n}$}
		\put(100,0){$10^{18}n^{-6}$}
		\end{picture}
		
		\subcaption{Plot of $\hat{v}_{n}=v_{n}-166.95208957$ vs.~$10^{18}n^{-6}$, where $v_{n}$ approximates the constant term $\gamma_{r}$ for relaxed trees.}
		\label{fig:relaxed_constant_approx_good}
	\end{subfigure}

	\begin{subfigure}{0.48\textwidth}
		\begin{picture}(190,190)
		\put(18,15){\includegraphics[width=\picwidthconst]{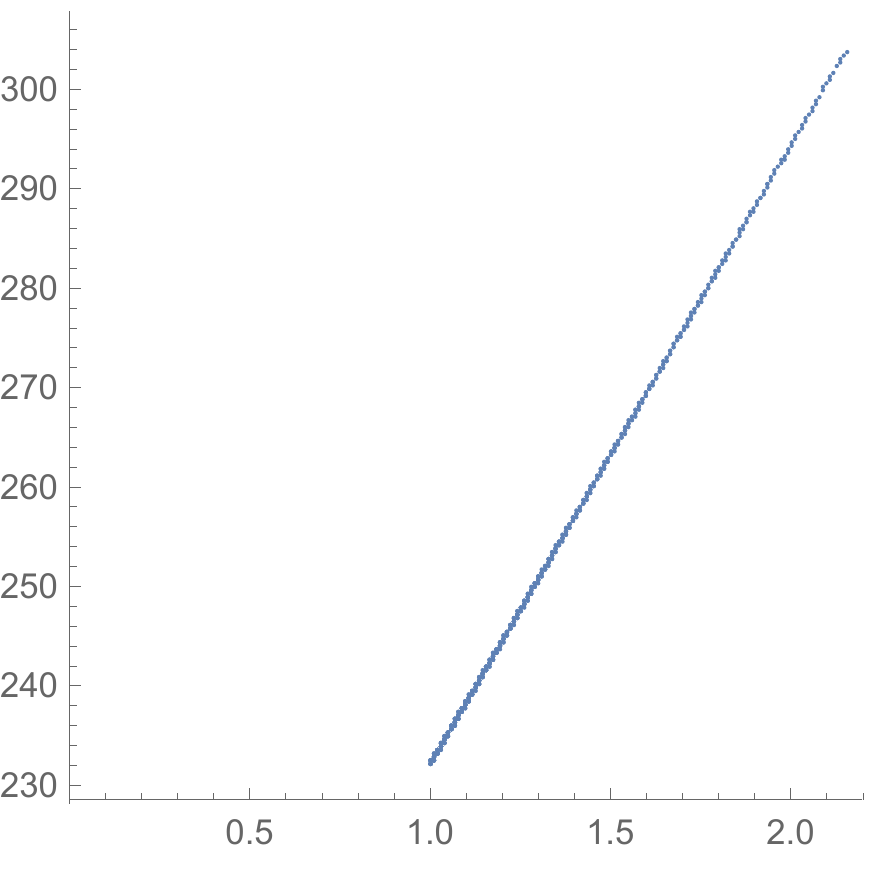}}
		\put(0,100){$u_{n}$}
		\put(100,0){$10n^{-1/3}$}
		\end{picture}
		
		\subcaption{Plot of $u_{n}$ vs.~$10n^{-1/3}$ for approximating the constant term $\gamma_{c}$ of compacted trees.\\}
		\label{fig:compacted_constant_approx_bad}
	\end{subfigure}\hfill
	\begin{subfigure}{0.48\textwidth}
		\begin{picture}(190,190)
		\put(18,15){\includegraphics[width=\picwidthconst]{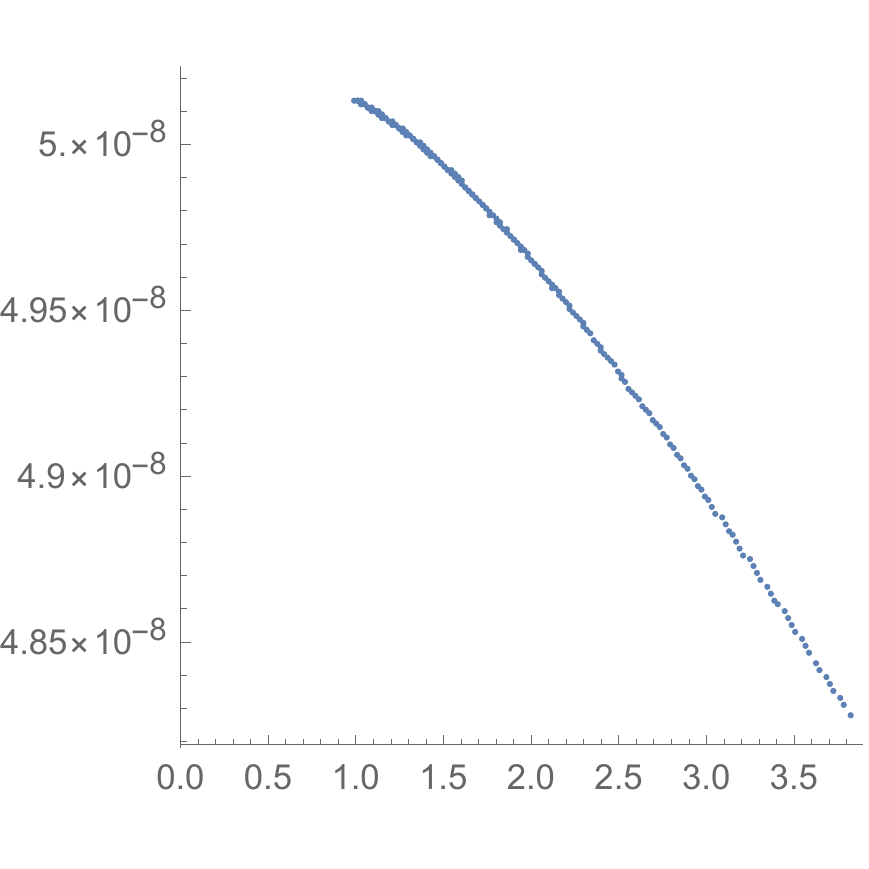}}
		\put(0,100){$\hat{v}_{n}$}
		\put(100,0){$10^{18}n^{-6}$}
		\end{picture}
		
		\subcaption{Plot of $\hat{v}_{n}=v_{n}-173.1267048$ vs.~$10^{18}n^{-6}$, where $v_{n}$ approximates the constant term $\gamma_{c}$ of compacted trees.}
		\label{fig:compacted_constant_approx_good}
	\end{subfigure}
	\caption{Plots for $800\leq n\leq1000$ visualizing the numerical approximation of the leading constants $\gamma_r$ and $\gamma_c$ of relaxed and compacted trees, respectively. Note that the scalings on the $x$-axes with $10n^{-1/3}$ and $10^{18}n^{-6}$ are chosen because $n$ is close to $1000$.}
\end{figure}
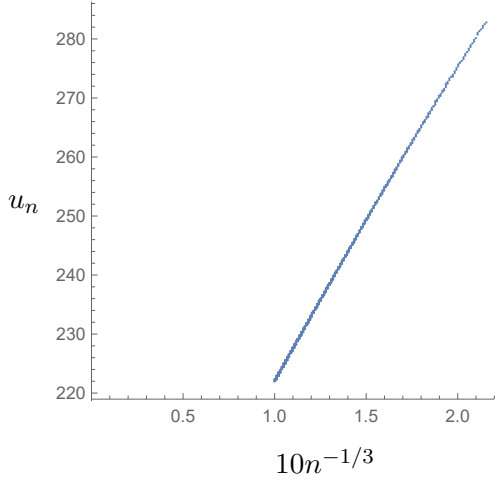
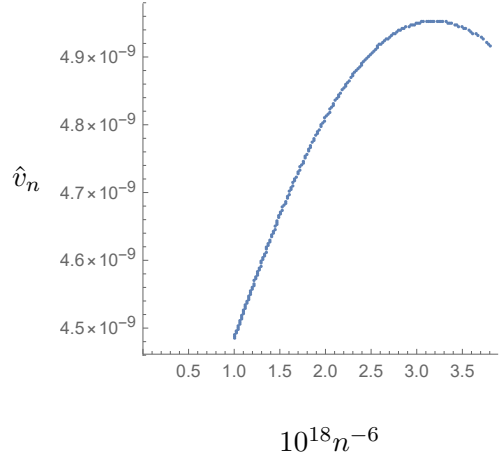
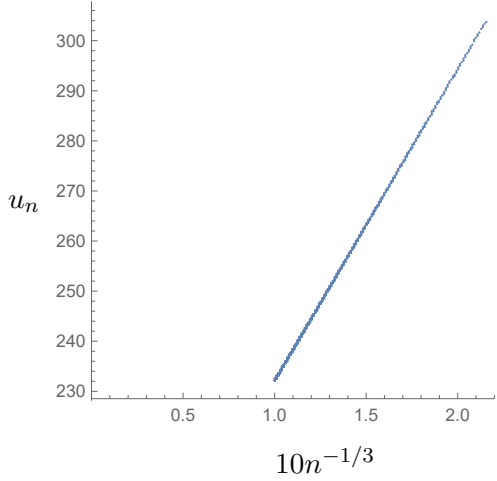
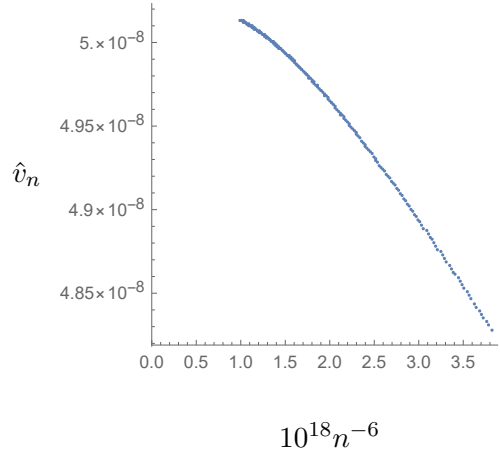

Based on the analysis in Section \ref{sec:heuristic-analytic}, we expect the ratios $r_{n}/r_{n-1}$ to behave like
\[\frac{r_{n}}{r_{n-1}}=\sum_{j=0}^{k-1}\beta_{j}n^{1-j/3}+\LandauO(n^{1-k/3}),\]
for any positive integer $k$, with the sequence $\beta_{0},\beta_{1},\ldots$ beginning with the terms $4,0,4a_{1},4$. This is equivalent to the existence of a sequence $\delta_{0},\delta_{1},\ldots$ such that $r_{n}$ behaves like
\[r_{n}=n!4^{n} \exp(3a_{1} n^{1/3})n\left(\sum_{j=0}^{k-1}\delta_{j}n^{-j/3}+\LandauO(n^{-k/3})\right),\]
for any positive integer $k$. In this equation, $\delta_{0}=\gamma_{r}$ is the constant term that we aim to approximate. A simple way to approximate $\gamma_{r}$ is to write
\[u_{n}=\frac{r_{n}}{n!4^{n} \exp(3a_{1} n^{1/3})n}.\] Then the graph of the values of $u_{n}$ plotted against $10n^{-1/3}$ (because $n$ is close to $1000$) should be roughly linear (see Figure~\ref{fig:relaxed_constant_approx_bad}), and the point where it crosses the $y$-axis can be taken as an approximation for $\gamma_{r}$. This yields $\gamma_{r}\approx160$. We get a more precise estimate as follows:
Fix $k$ to be some positive integer. Then, for each $n$, consider the integers $m\in[n,n+k)$. For each such $m$ we expect the equation
\[u_{m}\approx\sum_{j=0}^{k-1}\delta_{j}m^{-j/3}\]
to be approximately true. We then solve this system of equations for $\delta_{0},\ldots,\delta_{k-1}$ as though the equations were exact, using known, exact values of $u_{m}$. This yields approximations for 
$\delta_{0},\ldots,\delta_{k-1}$. Denote the approximation thus obtained for $\delta_{0} = \gamma_{r}$ by $v_{n}$. Note that this is equivalent to writing $v_{n}$ as a weighted sum of the numbers $u_{m}$, which cancels the terms $n^{-j/3}$ for $1\leq j<k$. For example, if $k=2$ then $v_{n}=((n+1)^{1/3}u_{n}-n^{1/3}u_{n+1})/((n+1)^{1/3}-n^{1/3})$. Hence, if our assumptions are correct then $v_{n}=\gamma_{r}+\LandauO(n^{-k/3})$. Taking $k=18$ and plotting $v_{n}$ against $10^{18}n^{-6}$ (because $n$ is close to $1000$) as in Figure \ref{fig:relaxed_constant_approx_good} yields the approximation $\gamma_{r}\approx166.95208957$, where we expect the quoted digits to be correct. In Figures \ref{fig:compacted_constant_approx_bad} and \ref{fig:compacted_constant_approx_good} we show a similar analysis of the counting sequence for compacted trees, yielding the approximation $\gamma_{c}\approx173.12670485$.

\section{Proof of stretched exponential for relaxed trees}
\label{sec:induction}
In this section we prove upper and lower bounds for the number of relaxed trees. These bounds differ only in the constant term, so they completely determine both the stretched exponential factor and the polynomial factor in the asymptotic number of relaxed trees for large $n$. 

Recall from Corollary \ref{coro:meander-in-dyck} that the number of relaxed trees $r_{n}$ of size $n$ is given by $r_{n}=n!\rs_{2n,0}$, where the terms $\rs_{n,m}$ are given by the recurrence relation~\eqref{eq:relaxedrecsimp} which we repeat here for the convenience of the reader:
\begin{align*}
	\left\{
		\begin{array}{rlrl}
		\rs_{n,m} &= \frac{n-m+2}{n+m} \rs_{n-1,m-1}+ \rs_{n-1,m+1}, & \text{ for } &n> 0, m\geq0,\\
		\rs_{0,m} &= 0, & \text{ for } & m>0, \\
		\rs_{n,-1} &= 0, & \text{ for } & n \geq 0, \\
		\rs_{0,0} &= 1.&&
		\end{array}
	\right.
\end{align*}

Our proofs of the upper and lower bounds for relaxed trees come from more general bounds for the numbers $\rs_{n,m}$, which we prove by induction. Suppose that $(X_{n,m})_{n\geq m\geq0}$ and $(s_{n})_{n\geq1}$ are sequences of non-negative real numbers satisfying
\begin{equation}\label{Xequation}
  X_{n,m}s_{n}\leq \frac{n-m+2}{n+m} X_{n-1,m-1}+ X_{n-1,m+1},
\end{equation}
for all sufficiently large $n$ and all integers $m\in[0,n]$. We define the sequence $(h_{n})_{n\geq0}$ by $h_{0}=1$ and $h_{n}=s_{n}h_{n-1}$. By induction on $n$, for some constant $b_0$, the following inequality
holds for all sufficiently large $n$ and all $m\geq0$:
\begin{align}
	X_{n,m}h_n &\stackrel{\text{\eqref{Xequation}}}{\leq} \frac{n-m+2}{n+m} X_{n-1,m-1} h_{n-1} + X_{n-1,m+1}h_{n-1} \notag \\
	           &\stackrel{\text{(IS)}}{\leq} \frac{n-m+2}{n+m} b_0 \rs_{n-1,m-1} + b_0 \rs_{n-1,m+1} \label{eq:inductionXH} \\
						 &\stackrel{\text{\eqref{eq:relaxedrecsimp}}}{=} b_0 \rs_{n,m}. \notag
\end{align}
Here (IS) marks the ``Induction Step''.
Similarly, if we can show the opposite of \eqref{Xequation}, it will imply that 
\[X_{n,m}h_{n}\geq b_1\cdot \rs_{n,m},\]
for all sufficiently large $n$ and all integers $m\in[0,n]$.

Comparing to the heuristic analysis in Section~\ref{sec:heuristic-analytic}, we see that $X_{n,m}$ acts as the function $f(\kappa)$, and $s_n$ as $s(n)$. Therefore, we should expect $X_{n,m}$ to be close to \eqref{eq:analytic-approx}, and $s_n$ to be a slight deviation of \eqref{eq:ratio-sn}.

In Lemma \ref{lem:AiryXLower} we will prove that certain explicit sequences $\tilde{X}_{n,m}$ and $\tilde{s}_{n}$ satisfy \eqref{Xequation}, which will lead to a lower bound on the numbers $\rs_{n,m}$. Similarly, in Lemma \ref{lem:AiryXUpper} we will show that other explicit sequences $\hat{X}_{n,m}$ and $\hat{s}_{n}$ satisfy the opposite of \eqref{Xequation}, which therefore yields an upper bound on the numbers $\rs_{n,m}$. Together, these two bounds determine the exact asymptotic form of the numbers $\rs_{2n,0}$ up to the constant term.

In order to prove these bounds with the explicit expressions of $X_{n,m}$ and $s_n$, we will consider the difference between the right- and the left-hand side of~\eqref{Xequation}. 
Then we will show that this difference is non-negative. 
We start by expanding the involved Airy function and its derivative in the neighborhood of an appropriate point $\alpha$, leading to a sum of the form
\[
  p_{n,m}\Ai(\alpha) + p'_{n,m}\Ai'(\alpha),
\]
where $p_{n,m}$ and $p'_{n,m}$ can be expressed as Puiseux series in $n$ whose coefficients are fractional polynomials in $m$. By looking at the ``Newton polygon'' of these Puiseux series, we can pick out the dominant term at different regimes of $n$ and $m$, leading to a proof of \eqref{Xequation} (or the reverse direction).

\begin{figure}[h]
	\centering
	\includegraphics[width=0.3\textwidth]{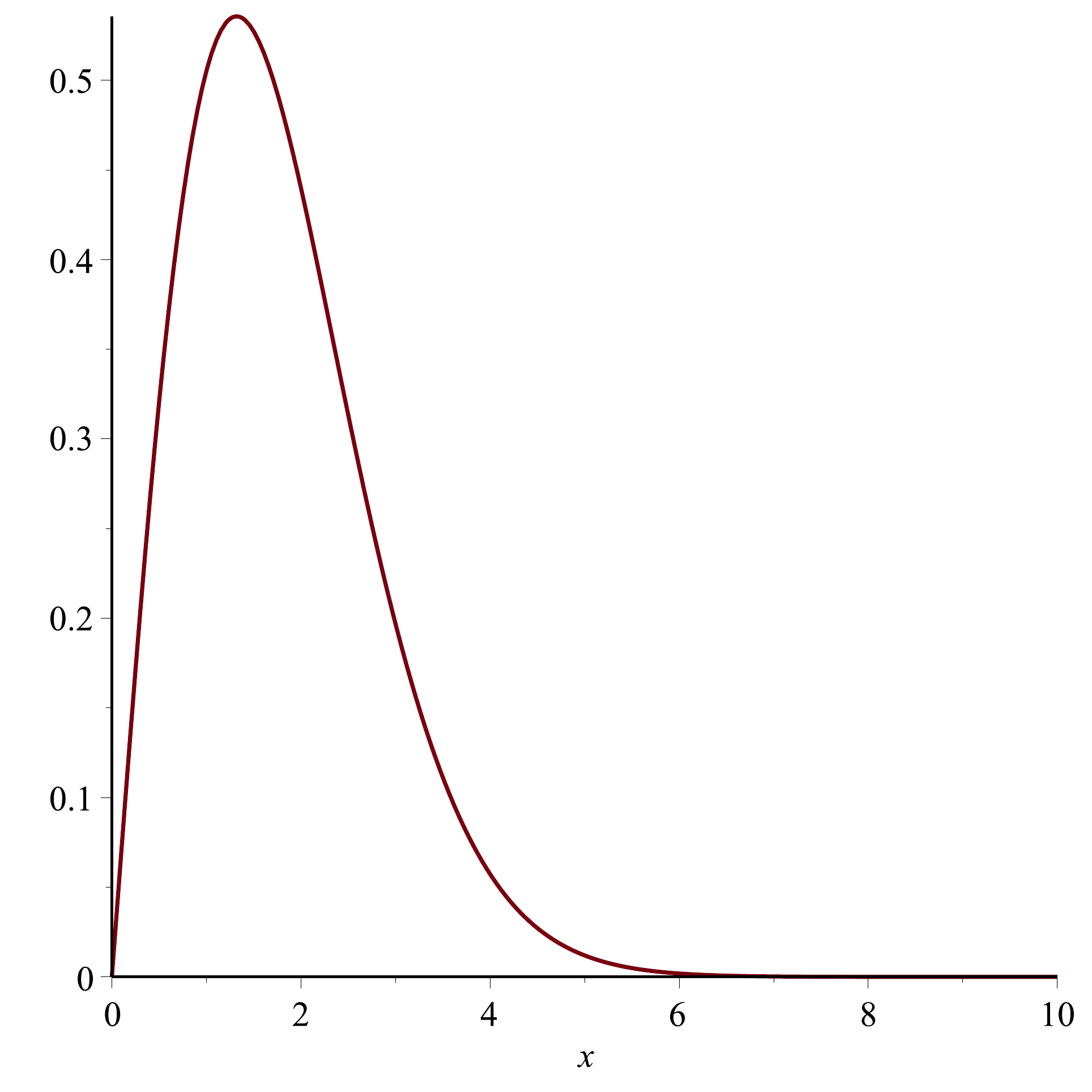}%
	\qquad
	\includegraphics[width=0.3\textwidth]{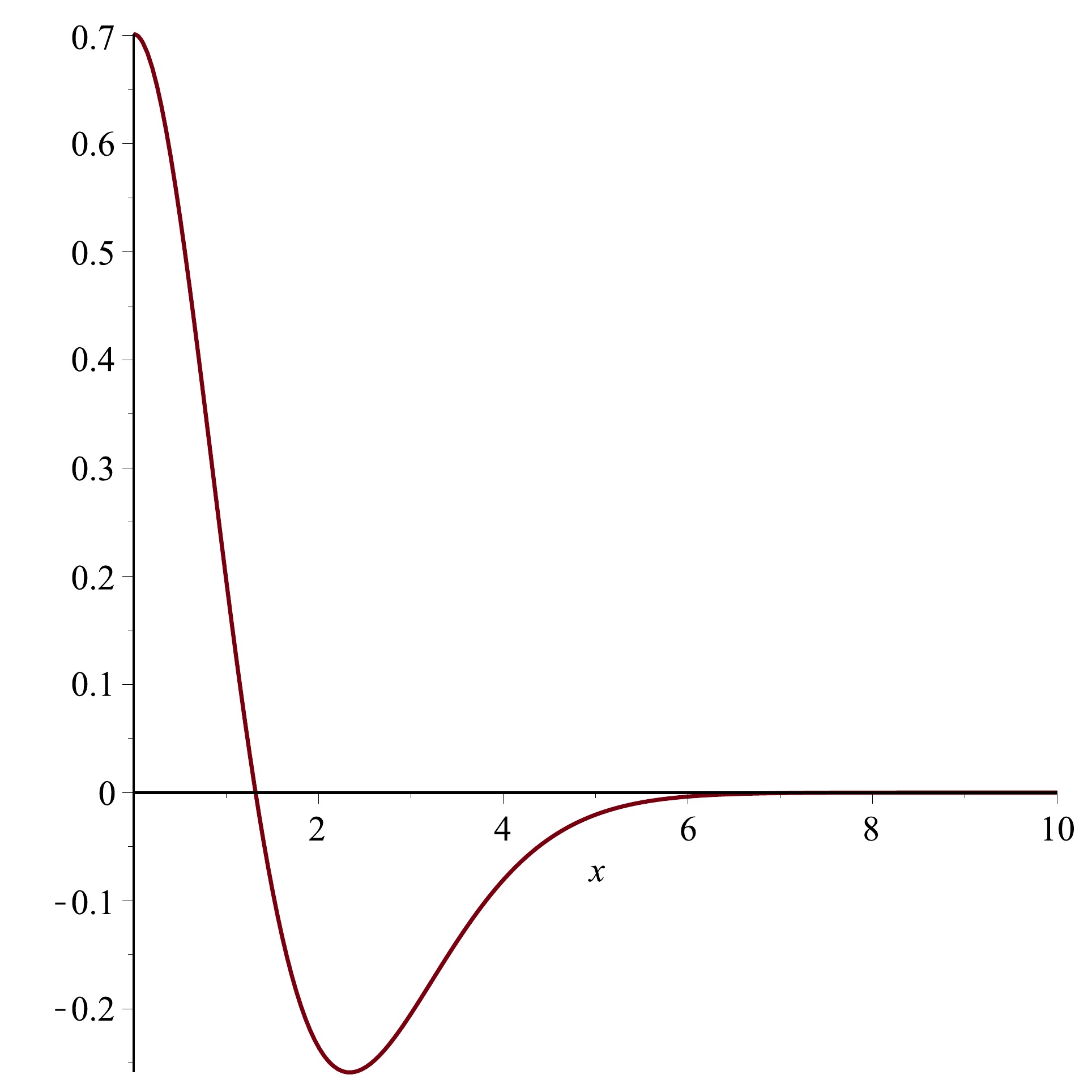}%
	\qquad 
	\includegraphics[width=0.3\textwidth]{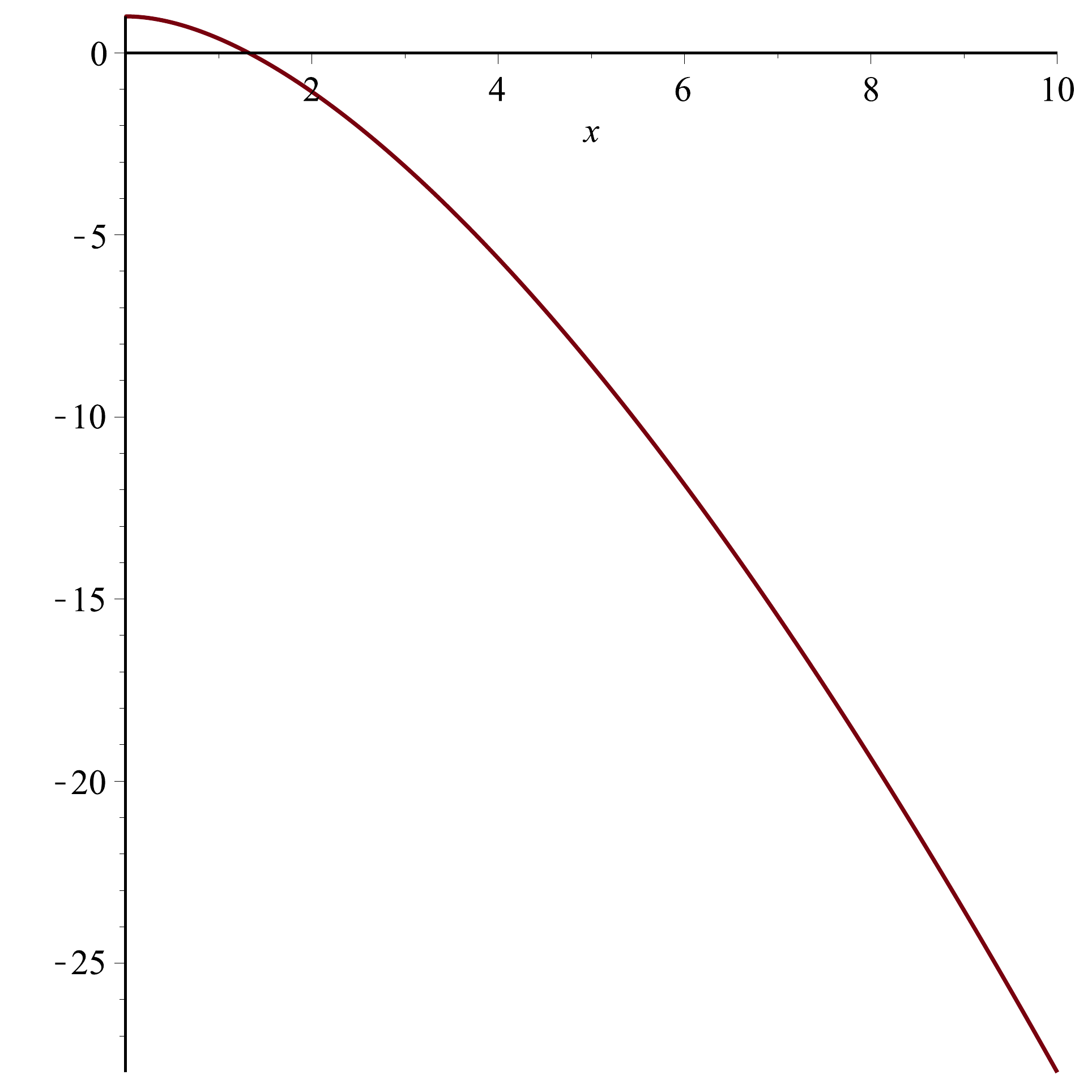}%
	\caption{(Left) The Airy function $\Ai(a_1+x)$, (Centre) its derivative $\Ai'(a_1+x)$, and (Right) the quotient $\Phi(x) = x\frac{\Ai'(a_1+x)}{\Ai(a_1+x)}$ on the positive real line.}
	\label{fig:AiryAi}
\end{figure}

The following Lemma summarizes some elementary results on the relation between the Airy function $\Ai$ and its derivative $\Ai'$. We will use these results in Lemmas \ref{lem:AiryXLower} and \ref{lem:AiryXUpper} to bound the subsequently defined auxiliary sequence $\tilde{X}_{n,m}$. 

\begin{lemma}
	\label{lem:PsiPhi}
	The functions
	\begin{align*}
		&&
		\Phi(x) &= x \frac{\Ai'(a_1+x)}{\Ai(a_1+x)}
		&& \text{and} &
		\Psi(x) &= \frac{\Ai'(a_1+x)}{\Ai(a_1+x)}
		&&
	\end{align*}
	are infinitely differentiable and monotonically decreasing on $x>0$ with $\Phi(0)=1$.
\end{lemma}

\begin{proof}
	First, by l'Hospital's rule it is easy to see that $\Phi(0)=1$. 
	Second, as $a_1$ is the largest root of $\Ai(x)$, the functions $\Phi(x)$ and $\Psi(x)$ are infinitely differentiable as compositions of differentiable functions.
	It remains to prove the monotonicity; see Figure~\ref{fig:AiryAi}. 
	A local expansion at $x=0$ shows that the functions are initially decreasing.
	The same holds for large $x$ due to the approximation $\Ai(x) \sim \frac{\exp\left(-\frac{2}{3}x^{3/2}\right)}{2\sqrt{\pi} x^{1/4}}$, 
	see~\cite[Equation~10.5.49]{AbramowitzStegun1964}, giving
	\begin{align}
		\label{eq:Psilarge}
		\Psi(x) \sim - \sqrt{a_1+x},
	\end{align}
	for $x \to \infty$. 
	We will show that $\Phi'(x)$ and $\Psi'(x)$ are always negative for $x>0$.
	Note that $\Phi(x)$ and $\Psi(x)$ will change sign only once at $x_0\approx0.91$.
	
	We present the following argument for the monotonicity of $\Phi(x)$. 
	Assume that there exists an $x_+$ such that $\Phi'(x_+)>0$. 
	Then, as $\Phi(x)$ is initially and finally decreasing, there must exist $y_1 < x_+ < y_2$ such that $\Phi'(y_1)=\Phi'(y_2)=0$ and $\Phi''(y_1) \geq 0 \geq \Phi''(y_2)$.	
	
	The second derivatives are equal to
	\[
	\Phi''(x) = 2a_1+3x - \frac{2}{x}\Phi(x)\Phi'(x).
	\]
	These lead to $2a_1 + 3 y_1 \geq 0 \geq 2a_1 + 3y_2$, thus also the contradiction $y_1 \geq y_2$.
	The argument for the monotonicity of $\Psi(x)$ is analogous, except that the second derivative is now
	\[
	\Psi''(x) = 1-2\Psi(x)\Psi'(x),
	\]
	leading to the contradiction $\Psi''(y_1)=\Psi''(y_2)=1$.
\end{proof}

Later we will use the value $x_0$ which is the unique root of $\Phi(x)$ and $\Psi(x)$ to determine the dominant term in the expansion of our series in $\Ai(x)$ and $\Ai'(x)$.

\subsection{Lower bound}
\label{sec:lowerbound}

\begin{lemma}
	\label{lem:AiryXLower}
	For all $n,m\geq0$ let
	\begin{align*}	
		\tilde{X}_{n,m} &:= \left(1-\frac{2m^2}{3n} + \frac{m}{2n}\right)\Ai\left(a_{1}+\frac{2^{1/3}(m+1)}{n^{1/3}}\right)~~~~~~~~\text{and}\\
		 \tilde{s}_n &:= 2+\frac{2^{2/3}a_1}{n^{2/3}}+\frac{8}{3n} - \frac{1}{n^{7/6}}.
	\end{align*}
	Then, for any $\varepsilon>0$, there exists an $\tilde{n}_0$ such that 
	\begin{align}
		\label{Xequation2}
		\tilde{X}_{n,m}\tilde{s}_{n} \leq \frac{n-m+2}{n+m} \tilde{X}_{n-1,m-1} + \tilde{X}_{n-1,m+1},
	\end{align}
	for all $n\geq \tilde{n}_0$ and for all $0 \leq m < n^{2/3-\varepsilon}$.
\end{lemma}

	\newcommand*{\aiarg}{\alpha}
	\newcommand{\lrvgen}{Z}

\begin{proof}
	First, define the following sequence
	\begin{align*}
		P_{n,m} := -\lrvgen_{n,m}{s}_{n} + \frac{n-m+2}{n+m} \lrvgen_{n-1,m-1} + \lrvgen_{n-1,m+1},
	\end{align*}
	where
    \begin{align*}
          s_n &:= \sigma_0 + \frac{\sigma_1}{n^{1/3}} + \frac{\sigma_2}{n^{2/3}}+\frac{\sigma_3}{n} + \frac{\sigma_4}{n^{7/6}}, \\
           \lrvgen_{n,m} &:= \left(1+\frac{\tau_2 m^2+\tau_1 m}{n}\right)\Ai\left(a_{1}+\frac{2^{1/3}(m+1)}{n^{1/3}}\right),
    \end{align*}
    with $\sigma_i, \tau_j \in\mathbb{R}$.
	Then the inequality~\eqref{Xequation2} is equivalent to $P_{n,m} \geq 0$ with $\sigma_0=2$, $\sigma_1=0$, $\sigma_2=2^{2/3}a_1$, $\sigma_3=8/3$, and $\sigma_4=-1$ as well as $\tau_0=0$, $\tau_1=1/2$, and $\tau_2=-2/3$.
	%
	%
	Next, we expand $\Ai(z)$ in a neighborhood of 
	\begin{align}
		\label{eq:Pnmexpansionlow}
		\aiarg = a_{1}+\frac{2^{1/3} m}{n^{1/3}},
	\end{align}
	and we get the following expansion
	\begin{align*}
		P_{n,m} &=  p_{n,m} \Ai(\aiarg) + p'_{n,m} \Ai'(\aiarg),
	\end{align*}
	where $p_{n,m}$ and $p'_{n,m}$ are functions of $m$ and $n^{-1}$ and may be expanded as power series in $n^{-1/6}$ with coefficients polynomial in $m$.
	As long as $n>1$ and $n>m$, this series converges absolutely because the Airy function is entire and so all functions expanded are analytic in the region defined by $|n|>1$ and $|n|>|m|$.

	As a first step we compute the possible range of the powers in $m$ and $n$.
	We will start by showing that $[m^i n^{j}]P_{n,m} = 0$ for $i+j > 1$, $i,j \in \Q$. The expansions of the three involved Airy functions only give terms of the form $\LandauO(m^jn^{-j}(n^{-1/3})^k)\Ai^{(k)}(\alpha)$, with $j,k\geq0$. Due to the differential equation $\Ai''(\alpha)=\alpha\Ai(\alpha)$, the term $\Ai^{(k)}(\alpha)$ takes the form $\LandauO(\alpha^{\lfloor k/2 \rfloor })\Ai(\alpha)+\LandauO(\alpha^{\lfloor (k-1)/2 \rfloor })\Ai'(\alpha)$. Hence, all terms in the expansion of the Airy function are of the form $\LandauO(m^j n^{-j})\Ai(\alpha)$ or $\LandauO(m^j n^{-j-1/3})\Ai'(\alpha)$ for some $j\geq0$. Due to the factor $m^2 n^{-1}$ in the definition of $\tilde{X}_{n,m}$, this implies that $[m^i n^{j}]P_{n,m} = 0$ for $i+j > 1$. Additionally, it also implies that the coefficients of $\Ai'(\alpha)$ are equal to $0$ for $i+j > 2/3$.
	
	Next, we strengthen this result by choosing suitable values $\sigma_i$ for $0 \leq i \leq 4$ in the definition of $s_n$ in order to eliminate more initial coefficients. 
	Then, we will show that the remaining terms satisfy $P_{n,m} \geq 0$.
	We performed this tedious task in Maple
	and we refer to the accompanying worksheet~\cite{Wallner2019web} for more details.
	The results are summarized in Figure~\ref{fig:PosP1} where the initial non-zero coefficients are shown.
	A diamond at $(i,j)$ is drawn if and only if the coefficient $[m^in^j]P_{n,m}$ is non-zero. 
	It is an empty diamond if the given choice of $\sigma_i$ and $\tau_j$ makes it disappear, whereas it is a solid diamond if it remains non-zero.
	The convex hull is formed by the following three lines 
	\begin{align*}
		L_1 &: j = -\frac{7}{6} - \frac{7i}{18}, \\
		L_2 &: j = -\frac{1}{3} - \frac{2i}{3}, \\
		L_3 &: j = 1 - i.
	\end{align*}
		
	\begin{figure}[ht]
		\centering
		\includegraphics[width=0.48\textwidth]{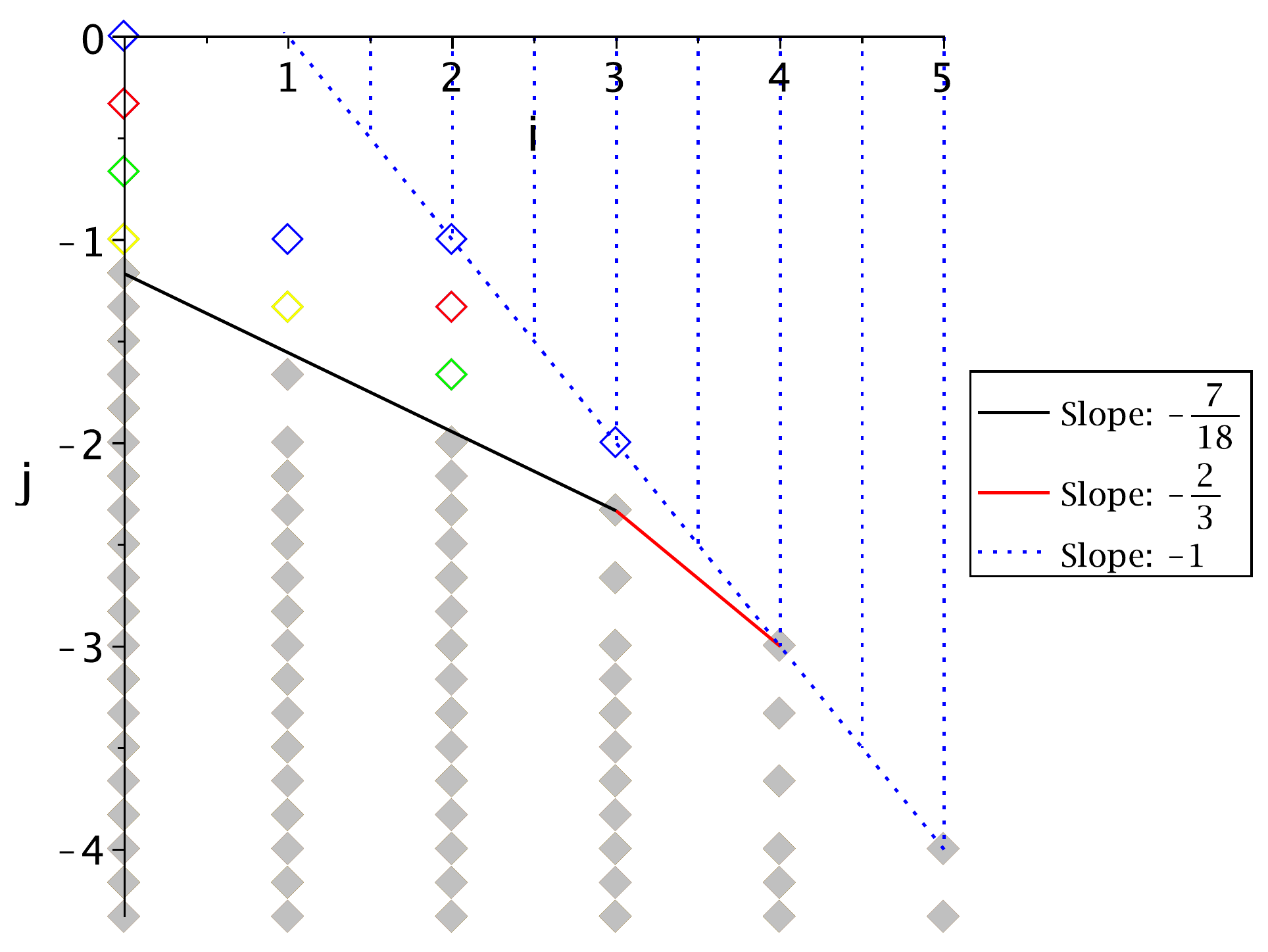}%
		\quad
		\includegraphics[width=0.48\textwidth]{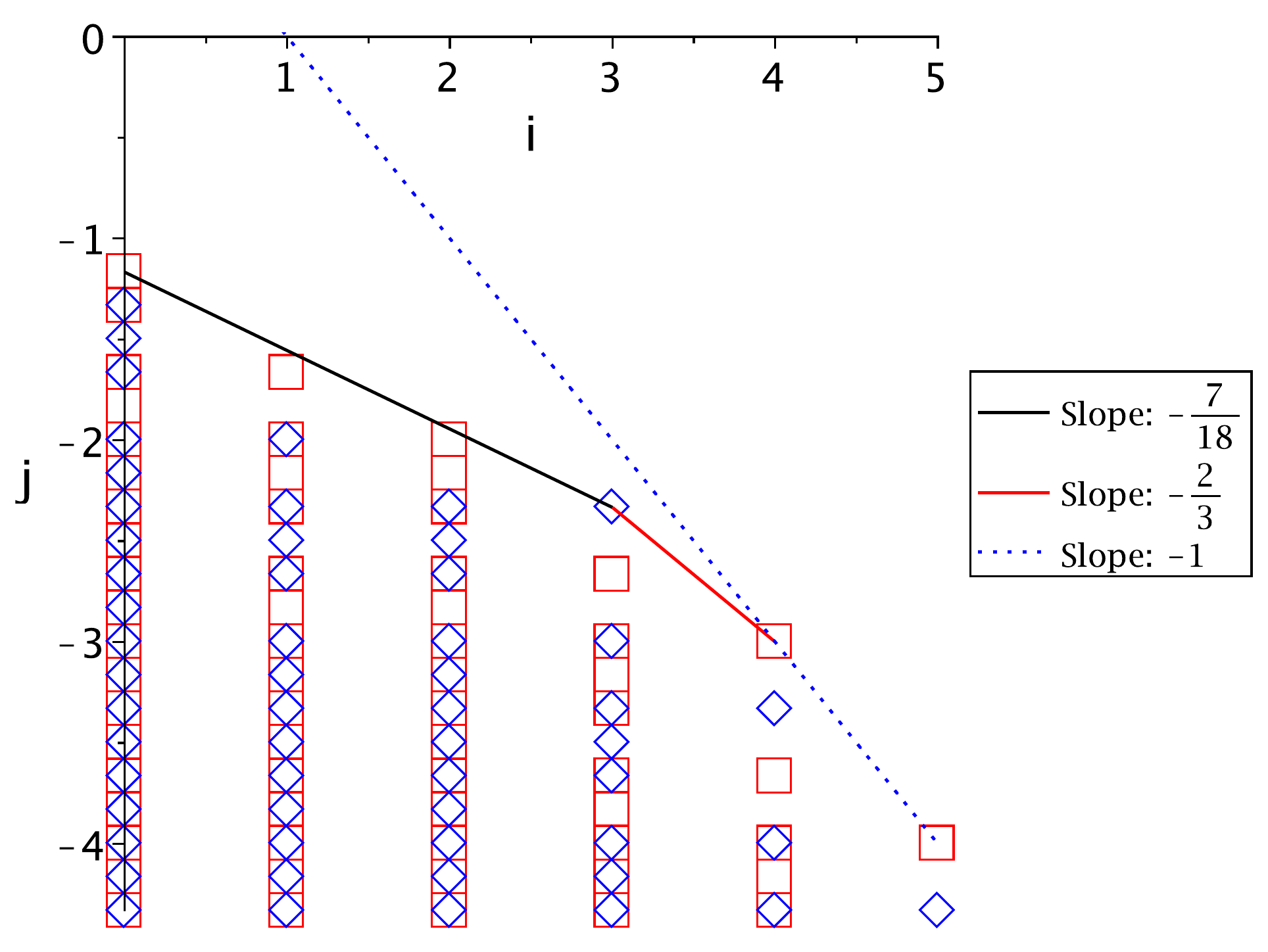}%
		\caption{(Left) Non-zero coefficients of $P_{n,m} = \sum {a_{i,j}} m^i n^j$ shown by diamonds for $s_n := \sigma_0 + \frac{\sigma_1}{n^{1/3}} + \frac{\sigma_2}{n^{2/3}}+\frac{\sigma_3}{n} + \frac{\sigma_4}{n^{7/6}}$ and $\lrvgen_{n,m} := \left(1+\frac{\tau_2 m^2+\tau_1 m}{n}\right)\Ai\left(a_{1}+\frac{2^{1/3}(m+1)}{n^{1/3}}\right)$. There are no terms in the blue dashed area. The blue terms vanish for $\sigma_0=2$, the red terms vanish for $\sigma_1=0$, the green terms vanish for $\sigma_2=2^{2/3}a_1$, and the yellow terms vanish for $\sigma_3=8/3$ and $\tau_2=-2/3$. The black and red lines represent the two parts $L_1$ and $L_2$, respectively, of the convex hull. (Right) The solid gray diamonds are decomposed into the coefficients $p_{n,m}$ of $\Ai(\alpha)$ (red boxes) and $p'_{n,m}$ of $\Ai'(\alpha)$ (blue diamonds).}
		\label{fig:PosP1}
	\end{figure}	
	
	Next, we distinguish between the contributions arising from $p_{n,m}$ and $p'_{n,m}$. 
	The expansions for $n$ tending to infinity start as follows, where the elements on the convex hull are written in color:
	\begin{align*}
		P_{n,m} =&~\Ai(\alpha) \left(
			\textcolor{red}{-\frac{\sigma_4}{n^{7/6}}} 
			- \frac{2^{5/3} a_1 m}{3 n^{5/3}} 
			\textcolor{red}{- \frac{41 m^2}{9n^2}}
			- \frac{2^{8/3} a_1 m^3}{3 n^{8/3}}
			\textcolor{red}{- \frac{34 m^4}{9n^3} - \frac{62 m^5}{135 n^4}} 
			+  \ldots
					\right) + \\
			    &~\Ai'(\alpha) \left(
			\textcolor{blue}{\frac{2^{1/3}(2\tau_1-1)}{n^{4/3}}}
			+ 
			\frac{2^{1/3}}{n^{3/2}}
			- \frac{8 a_1 m}{9n^2} 
			+ \frac{2^{1/3} (24 \tau_1 - 31) m^2}{9 n^{7/3}}
			\textcolor{blue}{- \frac{2^{13/3} m^3} {9n^{7/3}}} \right.\\
			&\qquad \qquad \left. \textcolor{blue}{-5 \frac{2^{5/3}m^4}{9n^{10/3}} -89 \frac{2^{4/3} m^5}{135 n^{13/3}}} 
			+ \ldots
			\right).
	\end{align*}
	We now choose $\sigma_4=-1$ which leads to a positive term $\Ai(\alpha)n^{-7/6}$ 
	and set $\tau_1=1/2$ to eliminate the term of order $n^{-4/3}$ from the convex hull (it is replaced by $\frac{2^{1/3}}{n^{3/2}}$).
	Then, the non-zero coefficients are shown in Figure~\ref{fig:PosP2}.
	Next, for fixed (large) $n$ we prove that for all $m$ the dominant contributions in $P_{n,m}$ are positive. 
	Therefore, we consider three different regimes. 	
	Let $x_0$ be the unique positive root of $\Psi(x)$ from Lemma~\ref{lem:PsiPhi}. 
	
	\begin{figure}[t]
		\centering
		\includegraphics[width=0.48\textwidth]{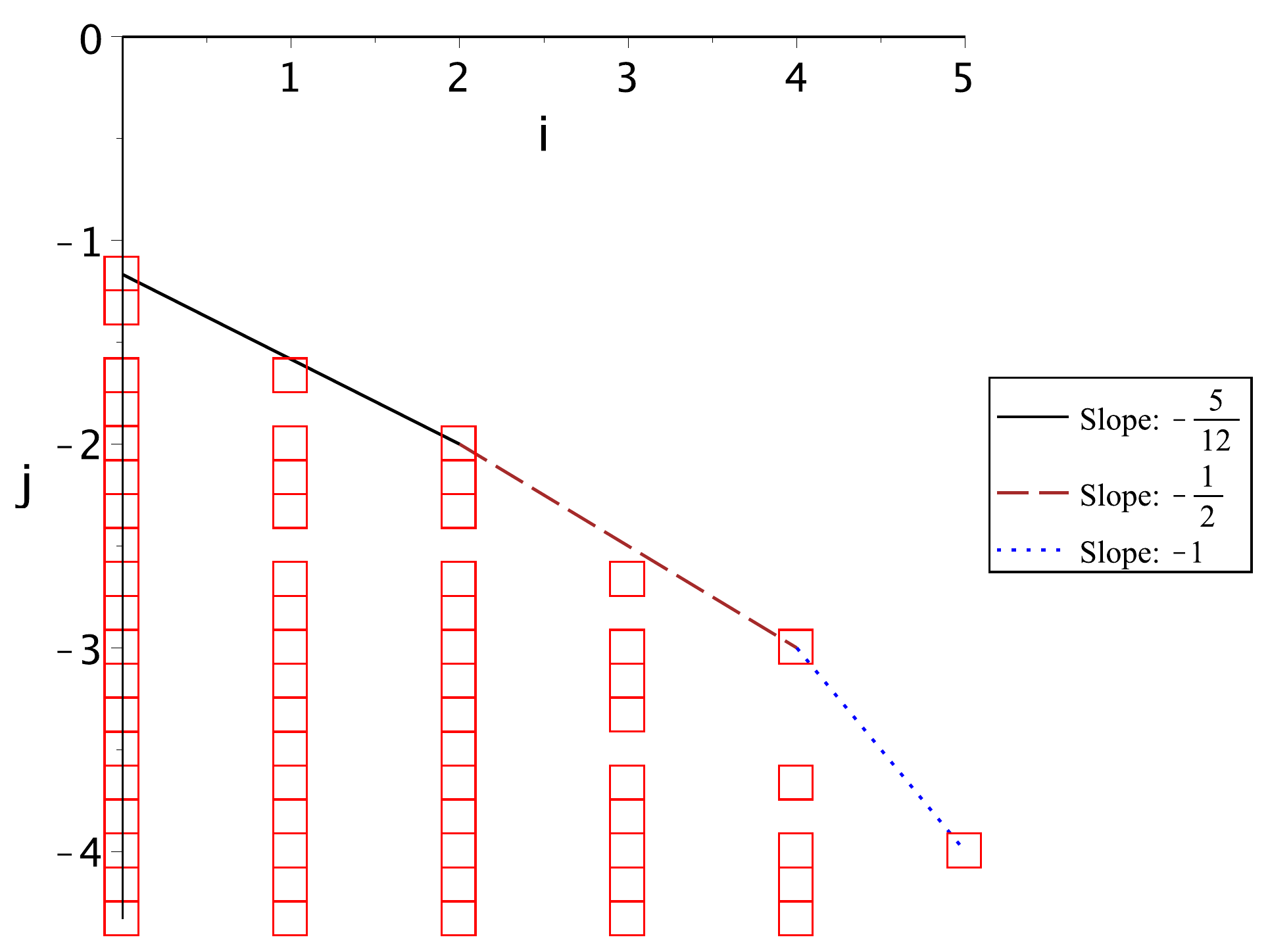}%
		\quad
		\includegraphics[width=0.48\textwidth]{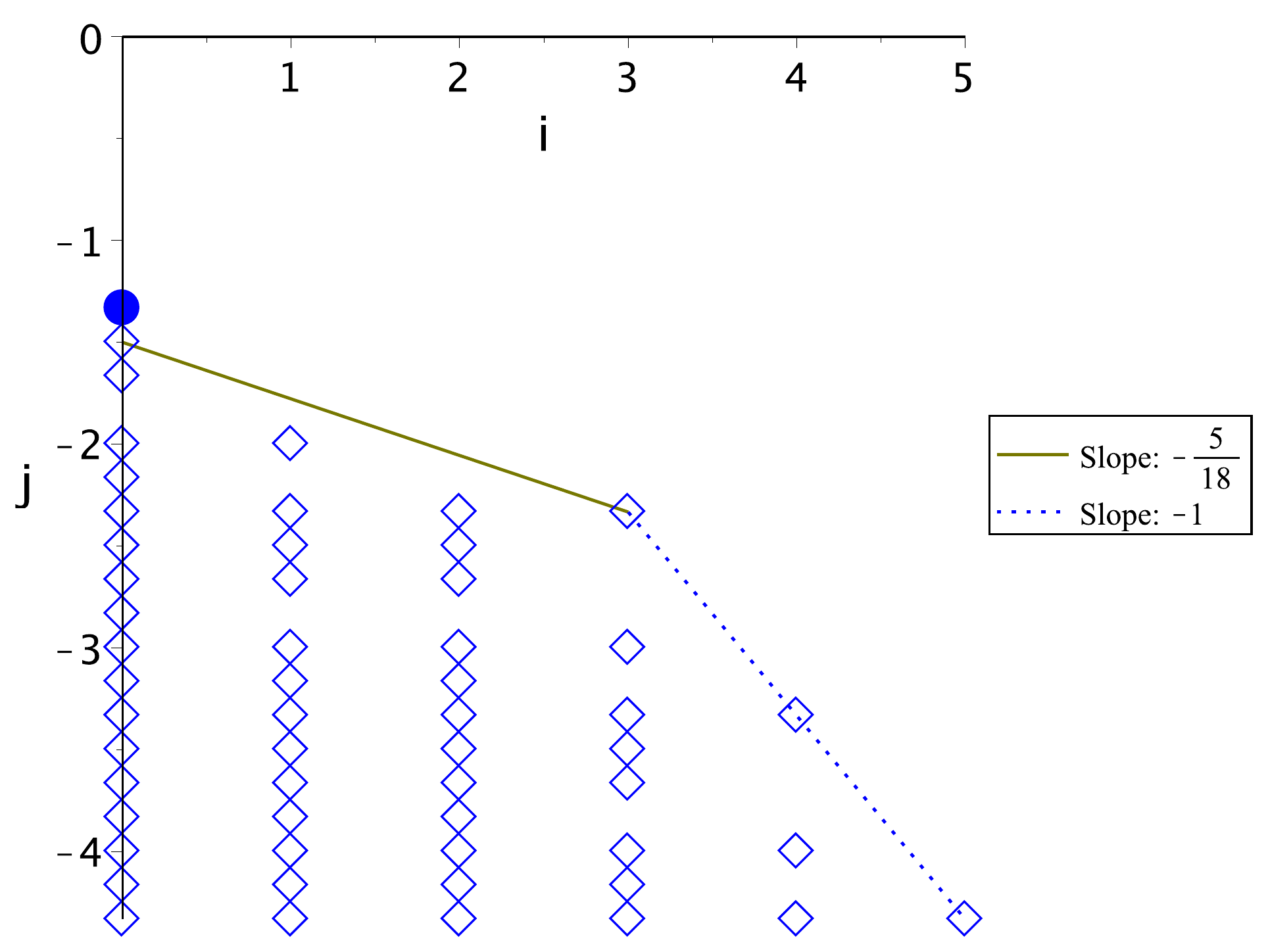}%
		\caption{Non-zero coefficients $p_{n,m} = \sum \tilde{a}_{i,j} m^i n^j$ (red) and $p'_{n,m} = \sum \tilde{a}_{i,j}' m^i n^j$ (blue) of the expansion~\eqref{eq:Pnmexpansionlow} for $P_{n,m}$. The coefficient of $n^{-4/3}$ in the right picture depicted as a solid blue circle disappears for $\tau_1=1/2$.}
		\label{fig:PosP2}
	\end{figure}
	
	\begin{enumerate}
		\item Consider the range of small values of $m$ given by $m \leq x_0 (n/2)^{1/3}$.
	In this range $\Ai(\alpha)$ and $\Ai'(\alpha)$ are both positive. 
	Moreover, the (red) coefficients of $\Ai(\alpha)$ are dominated by $n^{-7/6}$ for large $n$, while the (blue) coefficients of $\Ai'(\alpha)$ apart from the term $\nu=-\frac{2^{13/3}m^3}{9n^{7/3}}\Ai'(\alpha)$ are dominated by $\frac{2^{1/3}}{n^{3/2}}$. By Lemma~\ref{lem:PsiPhi} we have
	\begin{align*}
		\frac{2^{1/3} m}{n^{1/3}} \Ai'(\alpha) - \Ai(\alpha) < 0.
	\end{align*}
	Hence, $\nu  > -\frac{16 m^2}{9n^{2}}\Ai(\alpha)$,
	and it can therefore be treated as if it belonged to the coefficients of $\Ai(\alpha)$.
	Thus, as the dominating terms are positive, there exists some $N_{0}$ such that $P_{n,m} > 0$ whenever $n>N_{0}$ and $m \leq x_0 (n/2)^{1/3}$.
	
	\item Next, consider the central range $x_0 (n/2)^{1/3} < m \leq n^{7/18}$. 
	Here, we have $\Ai'(\alpha) < 0$.
	On the one hand, as seen in the left part of Figure~\ref{fig:PosP2}, the (red) coefficients of $\Ai(\alpha)$ are still dominated by $n^{-7/6}$ (which holds up to $m = \Theta(n^{5/12})$).	
	 On the other hand, in this range the term $\nu=-\frac{2^{13/3}m^3}{9n^{7/3}}\Ai'(\alpha)$ dominates all other (blue) coefficients of $\Ai'(\alpha)$ (due to $\tau_1=1/2$). Since $\nu>0$ in this range, this implies that there exists some (sufficiently large) $N_{1}$ such that $P_{n,m} > 0$ whenever $n>N_{1}$ and $x_0 (n/2)^{1/3}< m \leq n^{7/18}$.
	
	\item Finally, consider the range of large values $n^{7/18} < m <n^{2/3-\epsilon}$.
	By the reasoning on $\Psi(x)$ in Lemma~\ref{lem:PsiPhi} we see that $-\Ai'(\alpha) > \Ai(\alpha)>0$.
	Therefore, the (blue) term $\nu$ dominates all of the (red) terms of $\Ai(\alpha)$ as well as all other (blue) terms of $\Ai'(\alpha)$. Hence there exists some $N_{2}$ such that $P_{n,m} > 0$ whenever $n>N_{2}$ and $n^{7/18} < m <n^{2/3-\epsilon}$.
	\end{enumerate}
	
	Choosing $\tilde{n}_0=\max\{N_{0},N_{1},N_{2}\}$ completes the proof.
\end{proof}

\begin{remark}
	The previous result could be strengthened to hold up to $m \leq n^{1-\varepsilon}$ by~\eqref{eq:Psilarge} as will be shown in the proof of Lemma~\ref{lem:AiryXUpper}. However, we will not need this result in the sequel.
\end{remark}	

Now, to complete the lower bound we define the sequence $X_{n,m}:=\max\{\tilde{X}_{n,m},0\}$, i.e., %
\[
 X_{n,m}:=
  \begin{cases} 
      \hfill \tilde{X}_{n,m},    \hfill & \text{ if $m<\frac{\sqrt{96n+9}+3}{8}$,} \\
      \hfill 0,\hfill & \text{ if $m\geq \frac{\sqrt{96n+9}+3}{8}$.} \\
  \end{cases}
\]
Then, in the first case we get the following inequality for all sufficiently large $n$
\begin{align*}
	X_{n,m}\tilde{s}_{n}\leq \frac{n-m+2}{n+m} \tilde{X}_{n-1,m-1} + \tilde{X}_{n-1,m+1}\leq \frac{n-m+2}{n+m} X_{n-1,m-1} + X_{n-1,m+1},
\end{align*}
using Lemma \ref{lem:AiryXLower} with $\varepsilon=\frac{1}{12}$. Note that we could choose any $\varepsilon\in(0,\frac{1}{6})$, as we just need $n^{2/3-\varepsilon}>\frac{\sqrt{96n+9}+3}{8}$ for large $n$. In the second case we have
\[X_{n,m}\tilde{s}_{n}=0\leq \frac{n-m+2}{n+m} X_{n-1,m-1} + X_{n-1,m+1}.\]

Finally, we write $\tilde{h}_{n}=\tilde{s}_{n}\tilde{h}_{n-1}$ and we deduce by induction that $\rs_{n,m}\geq b \tilde{h}_{n}X_{n,m}$ for some constant $b>0$, all sufficiently large $n$ and all $m\in[0,n]$. In particular, it follows from~\eqref{eq:inductionXH} that the number $r_{n}=n!\rs_{2n,0}$ of relaxed trees of size $n$ is bounded below by
\begin{align}
	\label{eq:rnlower}
	r_{n} \geq \gamma\, n! 4^n e^{3a_1n^{1/3}} n^{},
\end{align}
for some constant $\gamma>0$. In the next section we will show an upper bound with the same asymptotic form, but with a different constant $\gamma$.

\subsection{Upper bound}

Next, we consider a similar auxiliary sequence $\hat{X}_{n,m}$ which will give rise to an upper bound on the number of relaxed binary trees.

\begin{lemma}
	\label{lem:AiryXUpper}
	Choose $\eta > 2/9$ fixed and for all $n,m \geq 0$ let
	\begin{align*}
		\hat{X}_{n,m} &:= \left(1-\frac{2m^2}{3n} + \frac{m}{2n} + \eta\frac{m^4}{n^2}\right)\Ai\left(a_{1}+\frac{2^{1/3}(m+1)}{n^{1/3}}\right)~~~~~~~~\text{and}\\
		\hat{s}_n &:= 2+\frac{2^{2/3}a_1}{n^{2/3}}+\frac{8}{3n} + \frac{1}{n^{7/6}}.
	\end{align*}
	Then, for any $\varepsilon>0$, there exists a constant $\hat{n}_0$ such that 
	\begin{align}
		\label{XequationB}
		\hat{X}_{n,m}\hat{s}_{n} \geq \frac{n-m+2}{n+m} \hat{X}_{n-1,m-1} + \hat{X}_{n-1,m+1},
	\end{align}
	for all $n\geq \hat{n}_0$ and all $0 \leq m < n^{1-\varepsilon}$.
\end{lemma}

\begin{proof}
	The proof follows the same lines as that of Lemma~\ref{lem:AiryXLower}. 
	Therefore we only focus on the needed modifications. 
	As a first step we define the following sequence
	\begin{align*}
		Q_{n,m} := \hat{X}_{n,m}\hat{s}_{n} - \frac{n-m+2}{n+m} \hat{X}_{n-1,m-1} - \hat{X}_{n-1,m+1}.
	\end{align*}
	Then the inequality~\eqref{XequationB} is equivalent to $Q_{n,m} \geq 0$. 
	Again, we expand $\Ai(z)$ in a neighborhood of 
	$
		\aiarg = a_{1}+\frac{2^{1/3} m}{n^{1/3}},
	$
	and we get (see~\cite{Wallner2019web} for more details)
	\begin{align*}
		Q_{n,m} &=  q_{n,m} \Ai(\aiarg) + q'_{n,m} \Ai'(\aiarg),
	\end{align*}
	where $q_{n,m}$ and $q'_{n,m}$ are functions of $m$ and $n^{-1}$ and may again be expanded as power series in $n^{-1/6}$ with coefficients polynomial in $m$.
	Now, it is easy to see that $[m^i n^{j}]Q_{n,m} = 0$ for $i+j > 2$, where the shift by~$1$ compared to the lower bound is due to the factor $\eta m^4n^{-2}$.
	The initial non-zero coefficients are shown in Figure~\ref{fig:PosP3}. 
	The four lines (black, red, green, blue) of the convex hull are
	\begin{align*}
		\hat{L}_1 &: j = -\frac{7}{6} - \frac{7i}{18}, \\
		\hat{L}_2 &: j = -\frac{5}{6} - \frac{i}{2}, \\
		\hat{L}_3 &: j = - \frac{2i}{3}, \\
		\hat{L}_4 &: j = 2 - i.
	\end{align*}
	
	\begin{figure}[ht]
		\centering
		\includegraphics[width=0.48\textwidth]{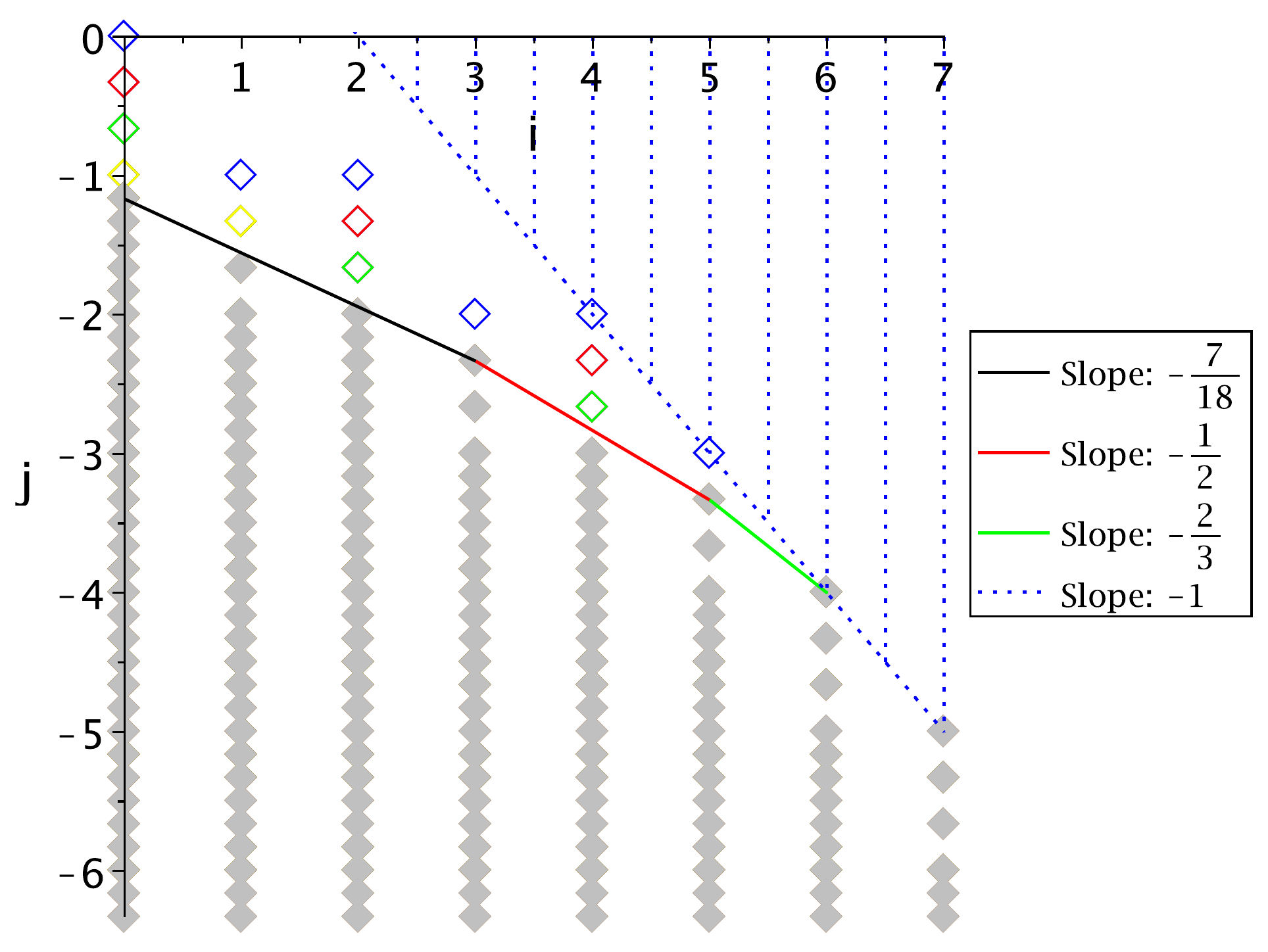}%
		\quad
		\includegraphics[width=0.48\textwidth]{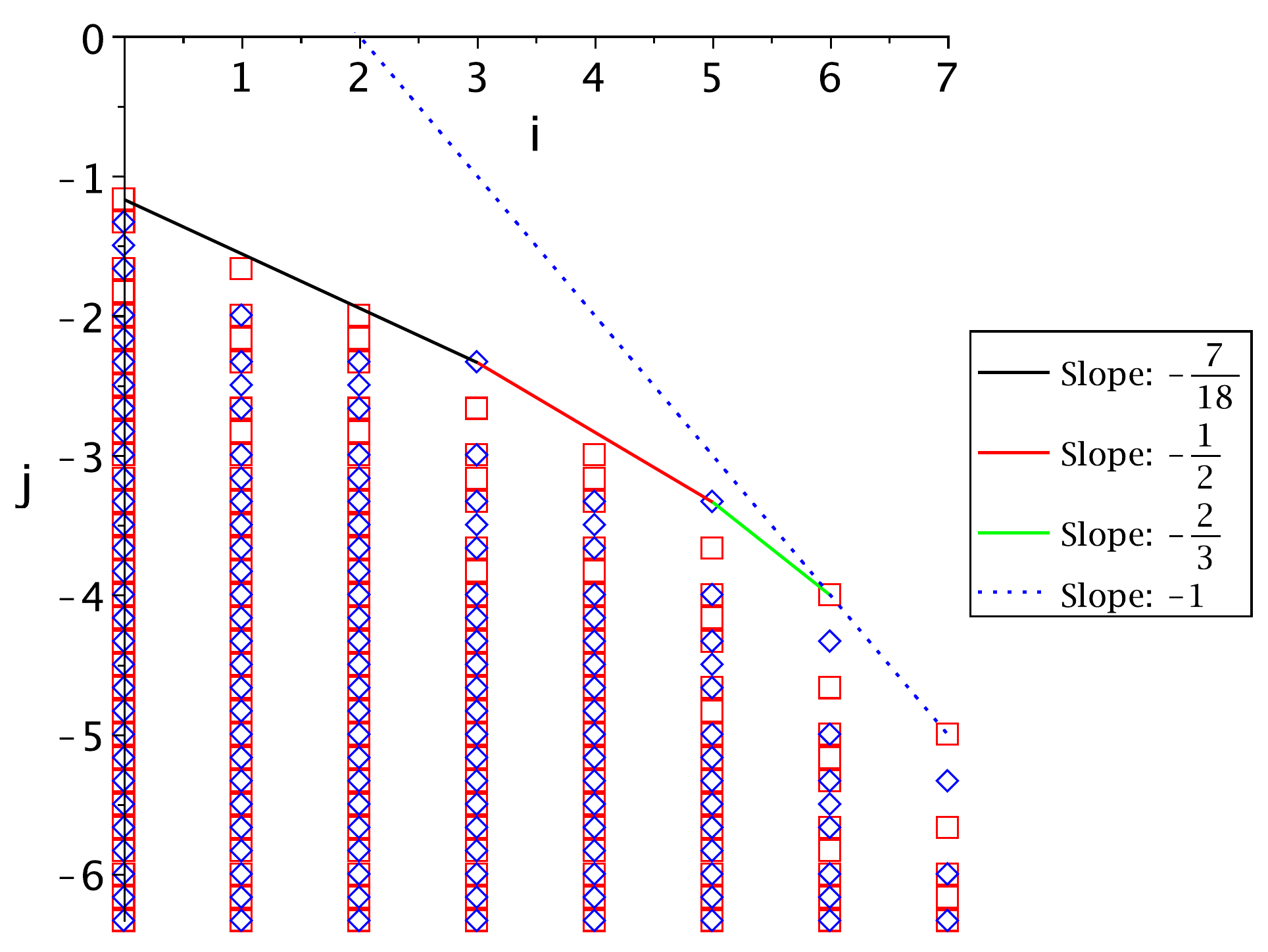}%
		\caption{(Left) Non-zero coefficients of $Q_{n,m} = \sum b_{i,j} m^i n^j$ shown in solid gray diamonds for $s_n := \sigma_0 + \frac{\sigma_1}{n^{1/3}} + \frac{\sigma_2}{n^{2/3}}+\frac{\sigma_3}{n} + \frac{\sigma_4}{n^{7/6}}$ and $X_{n,m} := \left(1+\frac{\tau_2 m^2+\tau_1 m}{n} + \eta \frac{m^4}{n^2} \right)\Ai\left(a_{1}+\frac{2^{1/3}(m+1)}{n^{1/3}}\right)$. There are no terms in the blue dashed area. The blue terms vanish for $\sigma_0=2$, the red terms vanish for $\sigma_1=0$, the green terms vanish for $\sigma_2=2^{2/3}a_1$, and the yellow term vanishes for $\sigma_3=8/3$ and $\tau_2=-2/3$. The black, red, and green lines represent the three parts $\hat{L}_1$, $\hat{L}_2$ and $\hat{L}_3$, respectively, of the convex hull. (Right) The solid gray diamonds are decomposed into the coefficients $q_{n,m}$ of $\Ai(\alpha)$ (red boxes) and $q'_{n,m}$ of $\Ai'(\alpha)$ (blue diamonds).}
		\label{fig:PosP3}
	\end{figure}

	\begin{figure}[ht]
		\centering
		\includegraphics[width=0.48\textwidth]{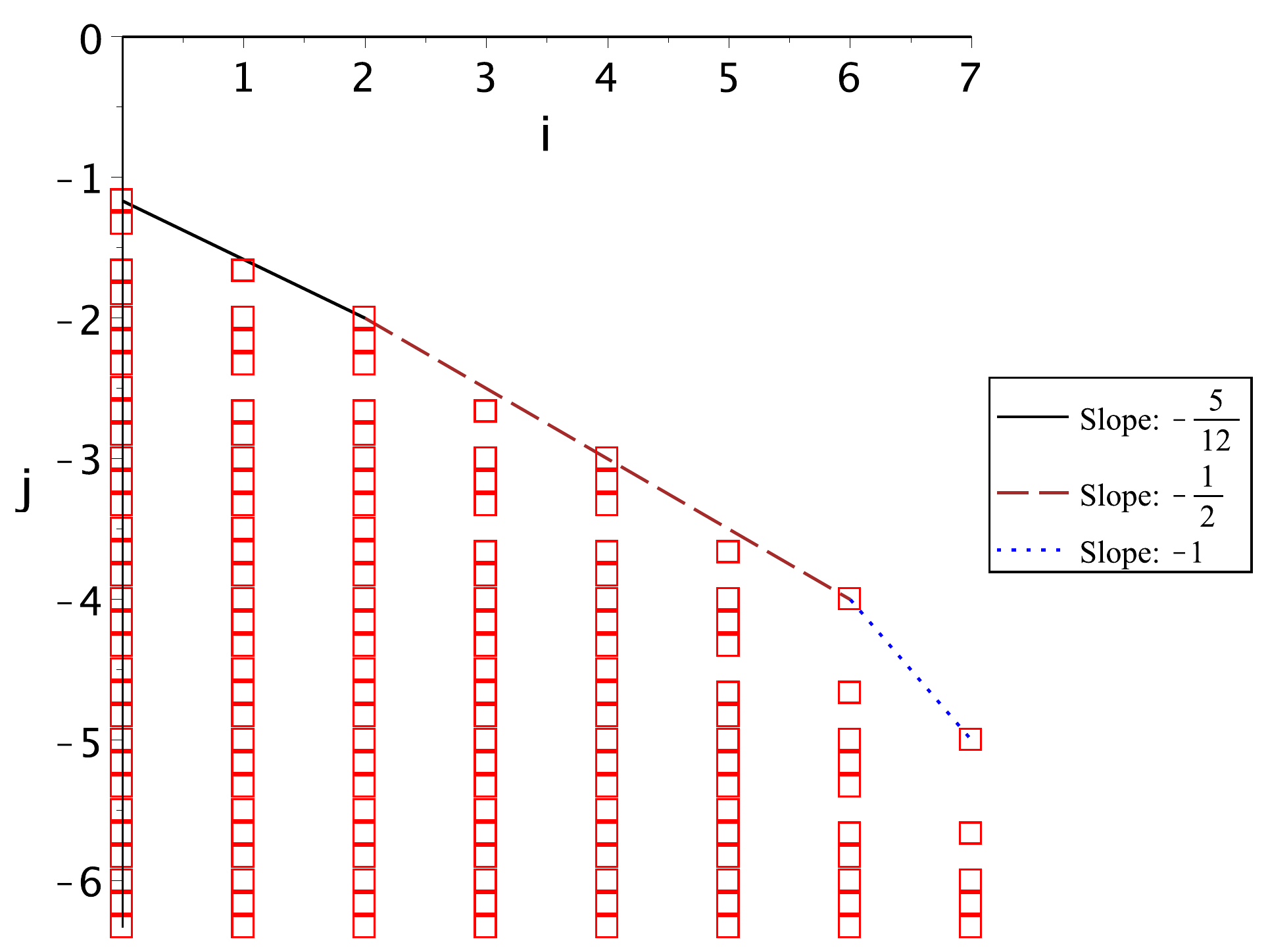}%
		\quad
		\includegraphics[width=0.48\textwidth]{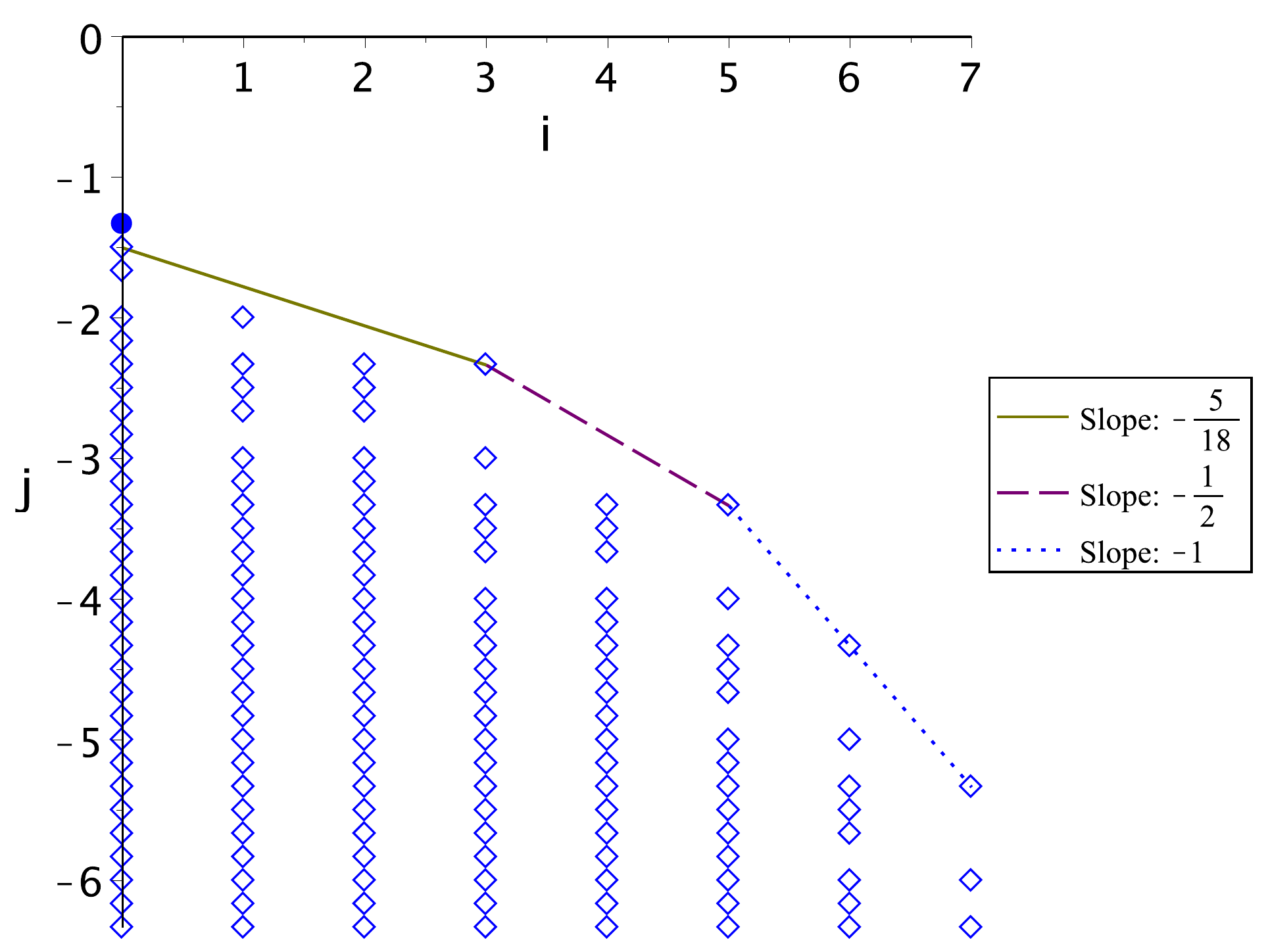}%
		\caption{Non-zero coefficients $q_{k,l} = \sum \tilde{b}_{i,j} m^i n^j$ (red) and $q'_{k,\ell} = \sum \tilde{b}_{i,j}' m^i n^j$ (blue) of the expansion for $Q_{n,m}$. The coefficient of $n^{-4/3}$ in the right picture depicted as a solid blue circle disappears for $\tau_1=1/2$.}
		\label{fig:PosP4}
	\end{figure}	
	
	\pagebreak
	Next, we distinguish between the contributions arising from $q_{n,m}$ and $q'_{n,m}$. The non-zero coefficients are shown in Figure~\ref{fig:PosP4}.
	The expansions for $n$ tending to infinity start as follows, where the elements on the convex hull are written in color.
	\begin{align*}
		Q_{n,m} =&~\Ai(\alpha) \left(
			\textcolor{red}{\frac{\sigma_4}{n^{7/6}}} 
			+ \frac{2^{5/3} a_1 m}{3 n^{5/3}} 
			\textcolor{red}{+ \frac{m^2 (41-108\eta)}{9n^2}} 
			+ \frac{2^{8/3} a_1 m^3(1-6\eta)}{3 n^{8/3}} 
			\textcolor{red}{+ \frac{2 m^4 (17-132\eta)}{9n^3}} \right. \\
			& \qquad \qquad - \left. \frac{2^{5/3} a_1 m^5 \eta}{n^{11/3}}  \textcolor{red}{- \frac{17 m^6 \eta}{3n^4}} \textcolor{red}{- \frac{31 m^7 \eta}{45 n^5}} +  \ldots
					\right) + \\
			    &~\Ai'(\alpha) \left(
			\textcolor{blue}{\frac{2^{1/3}}{n^{3/2}}} 
			+ \frac{8 a_1 m}{9n^2} 
			+ \frac{2^{1/3} m^2(19-108\eta)}{9 n^{7/3}} 
			\textcolor{blue}{+ \frac{2^{10/3} m^3(2-9\eta)}{9n^{7/3}}} + \frac{5m^4 2^{1/3} (2-27\eta)}{9n^{10/3}}  \right. \\
			& \qquad \qquad \left.  \textcolor{blue}{- \frac{2^{10/3} m^5 \eta}{3n^{10/3}} - \frac{5 m^6 2^{1/3} \eta}{3n^{13/3}} - \frac{89 m^7 2^{1/3} \eta}{45 n^{16/3}}} + \ldots
				    \right).
	\end{align*}	
	Let $x_0$ be again the unique positive root of $\Psi(x)$ from Lemma~\ref{lem:PsiPhi}.
	In order to prove that $Q_{n,m}\geq 0$ for $m \leq n^{1-\varepsilon}$, we consider the following four regions: 
	\begin{enumerate}
		\item $m \leq x_0 (n/2)^{1/3}$, 
		\item $x_0 (n/2)^{1/3} < m \leq n^{7/18}$,
		\item $n^{7/18} < m \leq n^{1/2}$,
		\item $n^{1/2} < m \leq n^{1-\varepsilon}$.
	\end{enumerate}
	Recall that in the proof that $P_{n,m}\geq0$ in Lemma \ref{lem:AiryXLower}, we considered almost the same first $3$ regions, except that in that case the upper bound on the third region was slightly larger ($n^{2/3-\epsilon}$). So the main difference here is the addition of the fourth region, which is required for this lemma to apply up to $m=n^{1-\varepsilon}$.
	
	The treatments of the first $3$ regions are analogous to those in Lemma~\ref{lem:AiryXLower} except for 2 minor changes.  
	First, in the second and third regime we include the additional variable $\eta$ to make the dominant term $\frac{2^{10/3} m^3(2-9\eta)}{9n^{7/3}}\Ai'(\alpha)$ positive. 
	Second, in the third regime an additional dominant term $- \frac{2^{10/3} m^5 \eta}{3n^{10/3}}\Ai'(\alpha)$ appears for $m = \Theta(n^{1/2})$ which is positive anyway.
	
	Finally, in the fourth regime, the aforementioned term $- \frac{2^{10/3} m^5 \eta}{3n^{10/3}}\Ai'(\alpha)$ is positive and dominates all other blue terms.
	However, the dominant red term is $- \frac{17 m^6 \eta}{3n^4}\Ai(z)$, which is negative, so it suffices to show that this is dominated by the blue term. 
	Indeed, due to~\eqref{eq:Psilarge} we know that as $m/n^{1/3}$ tends to infinity, $\Ai'(\alpha) \sim -2^{1/6} \frac{m^{1/2}}{n^{1/6}} \Ai(\alpha)$.
	Hence, the blue term $- \frac{2^{10/3} m^5 \eta}{3n^{10/3}} \Ai'(\alpha)$ dominates in this entire region $n^{1/2} < m \leq n^{1-\varepsilon}$.
\end{proof}

To finish the proof of the upper bound, we will choose some constant $N>0$ and define a sequence $\rsr_{n,m}$ by the same rules as $\rs_{n,m}$ except that $\rsr_{n,m}=0$ whenever $m>n^{3/4}$ and $n>N$. Then, writing $\hat{h}_{n}=\hat{s}_{n}\hat{h}_{n-1}$, we can use the lemma above to show by induction that the numbers $\rsr_{n,m}$ satisfy the inequality.
\[b_{0}\rsr_{n,m}\leq \hat{h}_{n}X_{n,m},\]
for some constant $b_{0}$ and all sufficiently large $n$; compare~\eqref{eq:inductionXH}. In particular, the numbers $\rsr_{2n,0}$ are bounded above by 
\[\rsr_{2n,0}\leq \gamma 4^{n}e^{3a_1n^{1/3}}n,\]
for some constant $\gamma>0$. The rest of this section is dedicated to proving that there is some choice of $N$ such that $\rsr_{2n,0}\geq \rs_{2n,0}/2$ for all $n$.

In order to finish our proof of the upper bound for the numbers $\rs_{2n,0}$, we will use the interpretation of these numbers as weighted Dyck paths, described in Section~\ref{sec:meander}. It will be useful to have an upper bound on the number of these paths which pass through a certain point $(2x,2y)$ as a proportion of the total weighted number of paths. 
Let $p_{\ell,m,2n}$ denote the weighted number of paths from $(\ell,m)$ to $(2n,0)$; see Figure~\ref{fig:sxyn}. 
Then the proportion $s_{x,y,n}$ of the $\rs_{2n,0}$ weighted Dyck paths that pass through $(2x,2y)$ is \[s_{x,y,n}=\frac{\rs_{2x,2y}p_{2x,2y,2n}}{\rs_{2n,0}}.\] The following lemma yields an upper bound on the number $p_{2x,2y,2n}$.

\begin{figure}[ht]
	\centering
	\includegraphics[width=0.5\textwidth]{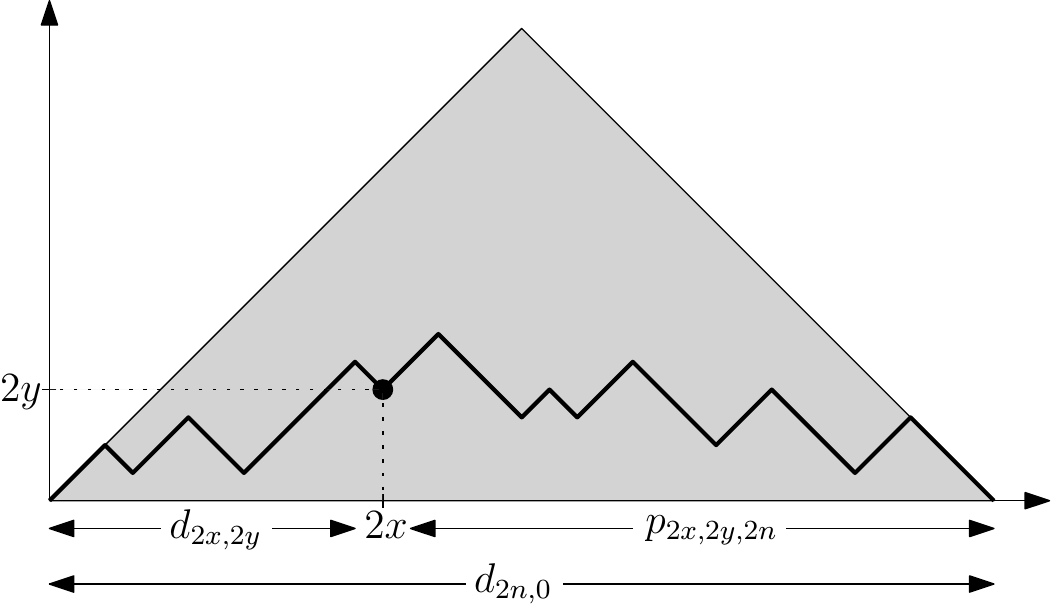}%
	\caption{Proportion of weighted Dyck paths of length $2n$ passing through the point $(2x,2y)$ showing one example path contributing to $s_{x,y,n}$.}
	\label{fig:sxyn}
\end{figure}	

\begin{lemma}\label{lem:relaxed_end} The numbers $p_{\ell,m,2n}$ satisfy the inequality
\[\frac{p_{\ell,j,2n}}{j+1}\geq \frac{p_{\ell,k,2n}}{k+1},\]
for integers $0\leq j<k\leq \ell\leq 2n$ satisfying $2 \mid k-j$.
\end{lemma}
\begin{proof}
First we note that the numbers $p_{\ell,m,2n}$ are determined by the recurrence relation
\[p_{\ell,m,2n}=p_{\ell+1,m-1,2n}+\frac{\ell-m+2}{\ell+m+2}p_{\ell+1,m+1,2n}\]
along with the initial conditions $p_{2n,m,2n}=\delta_{m,0}$ and $p_{l,-1,2n}=0$.
We will now prove the statement of the lemma by reverse induction on $\ell$.
Our base case is $\ell=2n$, for which the inequality clearly holds. For the inductive step, we assume that the inequality holds for $\ell+1$ and all $m$, and we will prove that it holds for $\ell$. It suffices to prove that for $m\geq 1$ the following inequality holds
\[\frac{p_{\ell,m-1,2n}}{m}-\frac{p_{\ell,m+1,2n}}{m+2}\geq0.\]
Let $L$ denote the left-hand side of this inequality. Using the recurrence relation, we can rewrite $L$ as
\[L=\frac{p_{\ell+1,m-2,2n}}{m}+\frac{(\ell-m+3)p_{\ell+1,m,2n}}{(\ell+m+1)m}-\frac{p_{\ell+1,m,2n}}{m+2}-\frac{(\ell-m+1)p_{\ell+1,m+2,2n}}{(\ell+m+3)(m+2)}.\]
Now, by the inductive assumption we get the inequalities $p_{\ell+1,m+2,2n}\leq\frac{m+3}{m+1}p_{\ell+1,m,2n}$ and $p_{\ell+1,m-2,2n}\geq\frac{m-1}{m+1}p_{\ell+1,m,2n}$, where the latter even holds for $m=1$ as then both sides are $0$. It follows that
\begin{align*}L&\geq \frac{(m-1)p_{l+1,m,2n}}{(m+1)m}+\frac{(\ell-m+3)p_{\ell+1,m,2n}}{(\ell+m+1)m}-\frac{p_{\ell+1,m,2n}}{m+2}-\frac{(m+3)(\ell-m+1)p_{\ell+1,m,2n}}{(m+1)(\ell+m+3)(m+2)}\\
&=\frac{4(3+m+\ell+m\ell)p_{\ell+1,m,2n}}{m(m+2)(1+m+\ell)(3+m+\ell)}\geq0\end{align*}
as desired. This completes the induction, which proves the inequality for $\ell\in[0,2n]$.
We refer to the accompanying worksheet~\cite{Wallner2019web} for more details.
\end{proof}

In particular, it follows from this lemma that
\[p_{2x,2y,2n}\leq (2y+1)p_{2x,0,2n}.\]
Moreover, note that the proportion $s_{x,0,n}$ of weighted paths passing through $(2x,0)$ satisfies $s_{x,0,n}\leq1$. Hence, the proportion $s_{x,y,n}$ satisfies
\begin{align}
	\label{eq:sxynupperbound}
	s_{x,y,n}=\frac{p_{2x,2y,2n}\rs_{2x,2y}}{\rs_{2n,0}}\leq\frac{(2y+1)p_{2x,0,2n}\rs_{2x,2y}}{\rs_{2n,0}s_{x,0,n}}=(2y+1)\frac{\rs_{2x,2y}}{\rs_{2x,0}}.
\end{align}
From the lower bound \eqref{eq:rnlower} we have
\[\rs_{2x,0}\geq\gamma 4^{x}e^{3a_1x^{1/3}}x,\]
so we now desire an upper bound for $\rs_{2x,2y}$. It will suffice to use the upper bound
\[\rs_{2x,2y}\leq{\binom{2x}{x+y}},\]
which holds because the right-hand side is the number of (unweighted) paths from $(0,0)$ to $(2x,2y)$, and all weights on our weighted paths are smaller than $1$. We are now ready to prove the following lemma
\begin{lemma}\label{lem:choose_const} 
For all $\varepsilon>0$ there exists a constant $N_{\varepsilon}>0$ with the following property:
Recall that $\rs_{n,m}$ is the weighted number of paths ending at $(n,m)$. Let $\rsr_{n,m}$ be the number of these paths such that no intermediate point $(2x,2y)$ on the path satisfies $x>N_{\varepsilon}$ and $y>x^{3/4}$. Then $\rs_{2n,0}\leq(1+\varepsilon)\rsr_{2n,0}$ for all $n>0$.
\end{lemma}
\begin{proof}
We can rewrite the desired inequality as
\[1-\frac{\rsr_{2n,0}}{\rs_{2n,0}}\leq\frac{\varepsilon}{1+\varepsilon}.\]
Note that the left-hand side is equal to the proportion of weighted paths with at least one intermediate point $(2x,2y)$ satisfying $x>N_{\varepsilon}$ and $y>x^{3/4}$. The proportion $s_{x,y,n}$ of weighted paths which go through any such point $(2x,2y)$ is bounded above by
\begin{align*}
	s_{x,y,n} 
		& \stackrel{\eqref{eq:sxynupperbound}}{\leq} (2y+1)\frac{\rs_{2x,2y}}{\rs_{2x,0}} \\
		& \stackrel{\eqref{eq:rnlower}}{\leq} \frac{2y+1}{\gamma 4^{x}e^{3a_1x^{1/3}}x}{\binom{2x}{x+y}} \\
		& \leq \gamma^{-1}4^{-x}e^{-3a_1x^{1/3}}x^{-1}\frac{(2y+1)\Gamma(2x+1)}{\Gamma(x+x^{3/4}+1)\Gamma(x-x^{3/4}+1)}.
\end{align*}
The right-hand side of this inequality behaves like
\begin{align}
	\label{eq:rhssxynbound}
	\Theta\left(e^{-x^{1/2}+\LandauO(x^{1/3})}\right)
\end{align}
for large $x$. Hence, there is some constant $c$ such that
\[s_{x,y,n}\leq c\cdot 2^{-x^{1/2}}\]
for all $x,y,n$ satisfying $y>x^{3/4}$.
Now, the proportion $1-\rsr_{2n,0}/\rs_{2n,0}$ of weighted paths passing through at least one point $(2x,2y)$ is no greater than the sum of the proportions of paths going through each such point. Hence
\begin{align*}
	1-\frac{\rsr_{2n,0}}{\rs_{2n,0}}\leq\sum_{x\geq N_{\varepsilon}+1^{\phantom{/}}}\!\sum_{x\geq y>x^{3/4}}s_{x,y,n}\leq\sum_{x\geq N_{\varepsilon}+1^{\phantom{/}}\!}\sum_{x\geq y>x^{3/4}}c\cdot 2^{-x^{1/2}}\leq\sum_{x\geq N_{\varepsilon}+1^{\phantom{/}}}\!cx\cdot 2^{-x^{1/2}}.
\end{align*}
The sum on the right converges to a value less than $\varepsilon/(1+\varepsilon)$ for sufficiently large $N_{\varepsilon}$. This completes the proof of the lemma.
\end{proof}

\begin{remark}
	Choosing $y>x^{\beta}$ instead of $y>x^{3/4}$ one can show that~\eqref{eq:rhssxynbound} behaves like $\LandauO\left(e^{-x^{2\beta-1}-3a_1x^{1/3}}\right)$. Hence, any $\beta>2/3$ gives the same result, yet $\beta=2/3$ is not sufficient.
\end{remark}

Finally, we define $\rsr_{n,m}$ as in Lemma~\ref{lem:choose_const} with some fixed $\epsilon > 0$. Then it follows from Lemma \ref{lem:AiryXUpper} that there is some constant $\gamma'>0$ 
 such that
\[\rsr_{2n,0}\leq \gamma' 4^{n}e^{3a_1n^{1/3}}n,\]
for all $n$. Hence
\[r_{n}=n!\rs_{2n,0}\leq 2\gamma' n!4^{n}e^{3a_1n^{1/3}}n,\]
completing the proof of the upper bound. We have now proven upper and lower bounds for the number $r_{n}$ of relaxed trees, which differ only in the constant term. Therefore,
\[r_{n}=\Theta\left(n!4^{n}e^{3a_1n^{1/3}}n\right).\]

\section{Proof of stretched exponential for compacted trees}
\label{sec:compacted}
\newcommand{\cs}{e}
\newcommand{\csr}{\hat{e}}
\newcommand{\csrr}{\tilde{e}}
\newcommand{\myarrayeq}{\!\!\!=}

We will now deal with compacted binary trees, whose recurrence as in Proposition~\ref{prop:reccompacted} has negative terms. We start by transforming the terms $c_{n,m}$ counting compacted trees to a sequence $\cs_{n,m}$ using the equation
\[
\cs_{n,m}=\frac{1}{((n+m)/2)!}c_{(n+m)/2,(n-m)/2},
\]
for $n-m$ even. Then, the terms $\cs_{n,m}$ are determined by the recurrence
\begin{align*}
	\left\{
		\begin{array}{rll}
		\cs_{n,m}&\myarrayeq \frac{n-m+2}{n+m} \cs_{n-1,m-1} + \cs_{n-1,m+1} - \frac{2(n-m-2)}{(n+m)(n+m-2)}\cs_{n-3,m-1}, & \text{for } n\geq m>0,\\
		\cs_{n,m}&\myarrayeq\cs_{n-1,m+1}, &\text{for } n> 0,m=0, \\
		\cs_{n,m} &\myarrayeq 1,&\text{for } n=m=0,\\
		\cs_{n,m} &\myarrayeq 0, & \text{for } m>n,
		\end{array}
	\right.
\end{align*}
and the number of compacted trees of size $n$ is equal to $n!e_{2n,0}$.

The method that we applied to \eqref{eq:relaxedrecsimp} in the relaxed case does not directly apply to this recurrence, as there is a negative term on the right-hand side. We solve this problem using the following lemma:

\begin{lemma}\label{lem:recurrence_bounds}
	For $n \geq 3$ and $n>m\geq0$, the term $\cs_{n,m}$ for compacted binary trees is bounded below by
	\begin{small}
	\[L_{\cs}=\frac{n-m+2}{n+m}\cs_{n-1,m-1} + \frac{n-m-2}{n-m} \cs_{n-1,m+1} + \frac{n-m-4}{n-m-2}\left( \frac{2}{n-m} \cs_{n-2,m+2} + \frac{2}{n+m}\cs_{n-3,m+1} \right)\]
	\end{small}
	and bounded above by
	\begin{small}
	\begin{align*}
	U_{\cs}&=
	\frac{n-m+2}{n+m}\cs_{n-1,m-1} + \frac{n-m-2}{n-m} \cs_{n-1,m+1} + \frac{2}{n-m} \cs_{n-2,m+2} \\
	& \quad + \frac{2}{n+m}\cs_{n-3,m+1}+\frac{4}{(n+m)(n+m-2)}\cs_{n-3,m-1}.
	\end{align*}\end{small}%
	That is, $L_{\cs}\leq\cs_{n,m}\leq U_{\cs}$. Furthermore, $U_\cs \leq U_\rs \leq \rs_{n,m}$ where $U_{\rs}$ is defined by the same expression as $U_{\cs}$ but with each $\cs$ replaced by $\rs$.
\end{lemma}

\begin{proof}
We start with the upper bound of $e_{n,m}$. In order to prove that, we will compute successively stronger upper bounds. We start with the trivial upper bound
\begin{align}
	\label{eq:bounde1}
	\cs_{n,m}\leq \frac{n-m+2}{n+m} \cs_{n-1,m-1} + \cs_{n-1,m+1}.
\end{align}
Applying this bound to $\cs_{n-1,m+1}$ then $\cs_{n-2,m}$ we find that
\begin{align}
	\label{eq:bounde2}
	\cs_{n-1,m+1}\leq \frac{n-m}{n+m}\left(\frac{n-m}{n+m-2} \cs_{n-3,m-1} + \cs_{n-3,m+1}\right)+ \cs_{n-2,m+2}.
\end{align}
Adding $2/(n-m)$ times this inequality to the defining equation of $e_{n,m}$ yields
\begin{align*}\cs_{n,m} &\leq
	\frac{n-m+2}{n+m}\cs_{n-1,m-1} + \frac{n-m-2}{n-m} \cs_{n-1,m+1}\\
	&~+ \frac{2}{n-m} \cs_{n-2,m+2} + \frac{2}{n+m}\cs_{n-3,m+1}+\frac{4}{(n+m)(n+m-2)}\cs_{n-3,m-1}.
	\end{align*}
	
	Now we will prove that $U_{\rs}\leq\rs_{n,m}$. Note that the first two inequalities~\eqref{eq:bounde1} and \eqref{eq:bounde2} in this proof become equalities when each $\cs$ is replaced by $\rs$. Adding $2/(n-m)$ times the latter (now) equality~\eqref{eq:bounde2} to the defining equation~\eqref{eq:relaxedrecsimp} of $\rs_{n,m}$ yields
	\begin{align*}
		U_{d}=\rs_{n,m} - \frac{2(n-m-2)}{(n+m)(n+m-2)} \rs_{n-3, m-1} \leq \rs_{n,m}.
	\end{align*}
	We then see that $\cs_{n,m} \leq U_{\cs} \leq U_{\rs} \leq \rs_{n,m}$ through induction on $n$.
	
	For the lower bound on $\cs_{n,m}$, we start with the inequality
\begin{align}
	\label{eq:elower1}
	\cs_{n,m}\geq\frac{n-m+2}{n+m}\cs_{n-1,m-1}.
\end{align}
This is clear for $m=0$, and for $m\geq n$ it is an equality. We can then deduce this inequality~\eqref{eq:elower1} for all $n,m$ using induction: Assume that the statement is true for all $n<N$ and all $m\in[0,n]$. Then, for $m\in[1,n-2]$ and $n=N$, we have
\[\frac{1}{n-m}\cs_{n-1,m+1}\geq\frac{1}{n+m}\cs_{n-2,m}
>\frac{n-m-2}{(n+m)(n+m-2)}\cs_{n-3,m-1}.\]
Hence,
\vspace{-3mm}
\begin{align*}\cs_{n,m}&=\frac{n-m+2}{n+m} \cs_{n-1,m-1} + \cs_{n-1,m+1} - \frac{2(n-m-2)}{(n+m)(n+m-2)}\cs_{n-3,m-1}\\
&\geq\frac{n-m+2}{n+m} \cs_{n-1,m-1} + \left(1-\frac{2}{n-m}\right)\cs_{n-1,m+1}.\\
&\geq\frac{n-m+2}{n+m} \cs_{n-1,m-1}.\end{align*}
This completes the induction. Moreover, it shows that
\vspace{-1mm}
\begin{align}
	\label{eq:lowere2}
	\cs_{n,m}\geq\frac{n-m+2}{n+m} \cs_{n-1,m-1} + \left(1-\frac{2}{n-m}\right)\cs_{n-1,m+1},
\end{align}
for $m\in[1,n-2]$. It is easy to see that this stronger inequality~\eqref{eq:lowere2} also holds for $m=0$ and $m\geq n$. Applying~\eqref{eq:lowere2} to $\cs_{n-1,m+1}$ then $\cs_{n-2,m}$ yields

\begin{align*}\frac{1}{n-m}\cs_{n-1,m+1}&\geq\frac{1}{n+m}\cs_{n-2,m}+\frac{n-m-4}{(n-m)(n-m-2)}\cs_{n-2,m+2}\\
&\geq\frac{1}{n+m}\left(\frac{n-m}{n+m-2}\cs_{n-3,m-1}+\frac{n-m-4}{n-m-2}\cs_{n-3,m+1}\right) \\
&\quad +\frac{n-m-4}{(n-m)(n-m-2)}\cs_{n-2,m+2}.\end{align*}
Finally, combining this with the inequality
\[\cs_{n,m}\geq\frac{n-m+2}{n+m} \cs_{n-1,m-1} + \cs_{n-1,m+1} - \frac{2(n-m)}{(n+m)(n+m-2)}\cs_{n-3,m-1}\]
yields the desired result.\end{proof}

The advantage of the bounds in the lemma above is that all terms are positive, so we can derive the asymptotics using the same techniques as for relaxed binary trees. 
Note that the behavior stays the same in the process of deriving the Newton polygons and leads to the same pictures as shown in Figures~\ref{fig:PosP1} and \ref{fig:PosP2}.

\vspace{-1mm}
\subsection{Lower bound}

The following result is analogous to Lemma~\ref{lem:AiryXLower}. 

\begin{lemma}
	\label{lem:AiryYLower}
	For all $n,m \geq 0$ let
	\begin{align*}
		\tilde{Y}_{n,m} &:= \left(1-\frac{2m^2}{3n} + \frac{m}{4n} \right)\Ai\left(a_{1}+\frac{2^{1/3}(m+1)}{n^{1/3}}\right)~~~~~~~~\text{and}\\
		\tilde{s}_n &:= 2+\frac{2^{2/3}a_1}{n^{2/3}}+\frac{13}{6n} - \frac{1}{n^{7/6}}.
	\end{align*}
	Then, for any $\varepsilon>0$, there exists a constant $\tilde{n}_0$ such that 
	\begin{align*}
		\begin{aligned}		
		\tilde{Y}_{n,m}\tilde{s}_{n}\tilde{s}_{n-1}\tilde{s}_{n-2} & \leq \frac{n-m+2}{n+m} \tilde{Y}_{n-1,m-1}\tilde{s}_{n-1}\tilde{s}_{n-2} + \frac{n-m-2}{n-m} \tilde{Y}_{n-1,m+1}\tilde{s}_{n-1}\tilde{s}_{n-2} \\
		& \quad + \frac{n-m-4}{n-m-2}\left( \frac{2}{n-m} \tilde{Y}_{n-2,m+2}\tilde{s}_{n-2}  + \frac{2}{n+m}\tilde{Y}_{n-3,m+1} \right),
		\end{aligned}
	\end{align*}
	for all $n\geq \tilde{n}_0$ and all $0 \leq m < n^{2/3-\varepsilon}$.
\end{lemma}

\begin{proof} 
	The proof is analogous to the case of relaxed trees.	
	In this case, the expansions for $n \to \infty$ start as follows, where the elements on the convex hull are written in color:
	\begin{align*}
		P_{n,m} =&~\Ai(\alpha) \left(
			\textcolor{red}{-\frac{4 \sigma_4}{n^{7/6}}} 
			- \frac{2^{11/3} a_1 m}{3 n^{5/3}} 
			\textcolor{red}{- \frac{164m^2}{9n^2}} 
			- \frac{2^{14/3} a_1 m^3}{3 n^{8/3}}
			\textcolor{red}{- \frac{136 m^4 }{9n^3} - \frac{248 m^5}{135n^4}}  +  \ldots
					\right) + \\
			    &~\Ai'(\alpha) \left( 
			\textcolor{blue}{\frac{2^{7/3}}{n^{3/2}}} 
			- \frac{32 a_1 m}{9n^2} 
			- 7\frac{2^{13/3} m^2}{9 n^{7/3}}
			\textcolor{blue}{- \frac{2^{19/3} m^3}{9n^{7/3}} - \frac{ 5 m^4 2^{10/3} }{9n^{10/3}} - \frac{89 m^5 2^{10/3}}{135 n^{13/3}}} 
			+ \ldots
				    \right).
	\end{align*}
	In this expansion we choose $\sigma_4=-1$, which leads to a positive term $\Ai(\alpha)n^{-7/6}$, and we also choose $\tau_1=1/4$ (instead of $1/2$ in the relaxed trees case), which kills the leading coefficient of $\Ai'(\alpha)2^{4/3}(4\tau_1-1)n^{-4/3}$ for small $m = \Landauo(n^{1/3})$.
	Then, the behavior and thus the pictures are identical to the case of relaxed trees shown in Figures~\ref{fig:PosP1}~and~\ref{fig:PosP2}.
	Hence, the proof follows exactly the same lines as that Lemma~\ref{lem:AiryXLower}.
\end{proof}
As in the relaxed case, we define a sequence $Y_{n,m}:=\max\{\tilde{Y}_{n,m},0\}$, i.e., %
\[
 Y_{n,m}:=
  \begin{cases} 
      \hfill \tilde{Y}_{n,m},    \hfill & \text{ if $m<\frac{\sqrt{384n+9}+3}{16}$,} \\
      \hfill 0,\hfill & \text{ if $m\geq \frac{\sqrt{384n+9}+3}{16}$.} \\
  \end{cases}
\]
Then defining $\tilde{h}_n = \tilde{s}_n \tilde{h}_{n-1}$, we get by induction 
$$\cs_{n,m} \geq \kappa_0 \tilde{h}_n Y_{n,m},$$
for some $\kappa_{0}>0$.  In particular, it follows that the number $c_{n}=n!\cs_{2n,0}$ of compacted trees of size $n$ is bounded below by
\begin{align}
	\label{eq:cnlower}
	c_{n} \geq \gamma\, n! 4^n e^{3a_1n^{1/3}} n^{3/4},
\end{align}
for some constant $\gamma>0$. In the next section we will show an upper bound with the same asymptotic form, but with a different constant $\gamma$.

\subsection{Upper bound}

The following result is analogous to Lemma~\ref{lem:AiryXUpper}.

\begin{lemma}
	\label{lem:AiryYUpper}
	Choose $\eta > 2/9$ fixed and for all $n,m \geq 0$ let
	\begin{align*}
		\hat{Y}_{n,m} &:= \left(1-\frac{2m^2}{3n} + \frac{m}{4n} + \eta\frac{m^4}{n^2}\right)\Ai\left(a_{1}+\frac{2^{1/3}(m+1)}{n^{1/3}}\right)~~~~~~~~\text{and}\\
		\hat{s}_n &:= 2+\frac{2^{2/3}a_1}{n^{2/3}}+\frac{13}{6n} + \frac{1}{n^{7/6}}.
	\end{align*}
	Then, for any $\varepsilon>0$, there exists a constant $\hat{n}_0$ such that 
	\begin{align*}
		\begin{aligned}
		\hat{Y}_{n,m}\hat{s}_{n}\hat{s}_{n-1}\hat{s}_{n-2} & \geq \frac{n-m+2}{n+m} \hat{Y}_{n-1,m-1}\hat{s}_{n-1}\hat{s}_{n-2} + \frac{n-m-2}{n-m} \hat{Y}_{n-1,m+1}\hat{s}_{n-1}\hat{s}_{n-2} \\
		& \quad + \frac{2}{n-m} \hat{Y}_{n-2,m+2}\hat{s}_{n-2}  \\
		& \quad + \frac{2}{n+m}\hat{Y}_{n-3,m+1} + \frac{4}{(n+m)(n+m-2)}\hat{Y}_{n-3,m-1},
		\end{aligned}
	\end{align*}
	for all $n\geq \hat{n}_0$ and all $0 \leq m < n^{1-\varepsilon}$.
\end{lemma}

\begin{proof}
	The proof is again analogous to the case of relaxed trees.		
	In this case, the expansions for $n \to \infty$ start as follows, where the elements on the convex hull are written in color:
	\begin{align*}
		Q_{n,m} =&~\Ai(\alpha) \left(
			\textcolor{red}{\frac{4 \sigma_4}{n^{7/6}}} 
			+ \frac{2^{11/3} a_1 m}{3 n^{5/3}} 
			\textcolor{red}{+ \frac{4m^2 (41-108\eta)}{9n^2}} 
			+ \frac{2^{14/3} a_1 m^3(1-6\eta)}{3 n^{8/3}} 
			\textcolor{red}{+ \frac{8 m^4 (17-132\eta)}{9n^3}} \right. \\
			& \qquad \qquad - \left. \frac{2^{11/3} a_1 m^5 \eta}{n^{11/3}}  \textcolor{red}{- \frac{68 m^6 \eta}{3n^4}} \textcolor{red}{- \frac{124 m^7 \eta}{45 n^5}} +  \ldots
					\right) + \\
			    &~\Ai'(\alpha) \left(
			\textcolor{blue}{\frac{2^{7/3}}{n^{3/2}}} 
			+ \frac{32 a_1 m}{9n^2} + \frac{2^{13/3} m^2(7-27\eta)}{9 n^{7/3}} \textcolor{blue}{+ \frac{2^{16/3} m^3(2-9\eta)}{9n^{7/3}}} + \frac{2^{4/3} m^4  (20-279\eta)}{9n^{10/3}}  \right. \\
			& \qquad \qquad \left. \textcolor{blue}{- \frac{2^{16/3} m^5 \eta}{3n^{10/3}} - \frac{5 m^6 2^{7/3} \eta}{3n^{13/3}} - \frac{89 m^7 2^{7/3} \eta}{45 n^{16/3}}} + \ldots
				    \right).
	\end{align*}
	In this expansion we choose $\sigma_4=1$, which leads to a positive term $\Ai(\alpha)n^{-7/6}$, and again $\tau_1 = 1/4$.
	Then, the behavior and therefore the pictures are identical to the case of relaxed trees shown in the Figures~\ref{fig:PosP3} and \ref{fig:PosP4}; see the proof of Lemma~\ref{lem:AiryXUpper} for more details.
\end{proof}

As in the relaxed tree case, the inequality of Lemma~\ref{lem:AiryYUpper} is only proven for $m<n^{1-\varepsilon}$, so we need to do more work to handle the $m\geq n^{1-\varepsilon}$ case and deduce the desired upper bound. In order to use the lemma, we define a new sequence $\csr_{n,m}$ by the recurrence relation
\begin{align*}
	\left\{
		\begin{array}{rlrl}
		\csr_{n,m}&= \frac{n-m+2}{n+m}\csr_{n-1,m-1} + \frac{n-m-2}{n-m} \csr_{n-1,m+1}&\\
		 &\quad + \frac{2}{n-m} \csr_{n-2,m+2} + \frac{2}{n+m}\csr_{n-3,m+1}&\\
		 &\quad+ \frac{4}{(n+m)(n+m-2)}\csr_{n-3,m-1}, & \text{ for } n\geq 3, n>m\geq0,\\
		\csr_{n,m} &= \cs_{n,m}, & \text{ otherwise.}
		\end{array}
	\right.
\end{align*}
Then it follows from Lemma~\ref{lem:recurrence_bounds} that $\cs_{n,m}\leq\csr_{n,m}\leq\rs_{n,m}$ for all $n,m$, as $\csr_{n,m}$ share the same recurrence as $U_d$ in Lemma~\ref{lem:recurrence_bounds}.
Now consider some large $N>0$, to be determined later, and define a second sequence $\csrr_{n,m}$ by the same rules as $\csr_{n,m}$ except that $\csrr_{n,m}=0$ whenever $m>n^{3/4}$ and $n>N$. Then, using Lemma \ref{lem:AiryYUpper} and defining $\hat{h}_{n}=\hat{s}_{n}\hat{h}_{n-1}$, we can show by induction that there is some constant $\kappa_{1}$ such that
\[\csrr_{n,m}\leq\kappa_{1}\hat{h}_{n}\hat{Y}_{n,m}.\]
It follows that there is some constant $\gamma'>0$ such that
\[\csrr_{2n,0}\leq \gamma'4^{n}e^{3a_1n^{1/3}}n^{3/4}.\]
Hence, it suffices to prove that there is some choice of $N$ and some constant $\varepsilon>0$ such that $\csr_{2n,0}\leq (1+\epsilon)\csrr_{2n,0}$ for all $n$.
Therefore, we first define a class $\Cc$ of weighted paths with the step set $\{(1,1), (1,-1), (2,-2), (3,-1), (3,1)\}$ and weights corresponding to the recurrence defining $\csr_{n,m}$. Then $\csr_{n,m}$ is the weighted number of paths $p\in\Cc$ from $(0,0)$ to $(n,m)$. We start with the following lemma, which is analogous to Lemma~\ref{lem:relaxed_end}.

\begin{lemma}\label{lem:compact_end} Let $q_{\ell,m,2n}$ denote the weighted number of paths $p\in\Cc$ from $(\ell,m)$ to $(2n,0)$. Then the numbers $q_{\ell,m,2n}$ satisfy the inequality
\[\frac{q_{\ell,j,2n}}{j+1}\geq \frac{q_{\ell,k,2n}}{k+1},\]
for integers $0\leq j<k\leq \ell\leq 2n$ satisfying $2|k-j$ and $n\geq10$.
\end{lemma}
\begin{proof}
The proof is along the same lines as the proof of Lemma \ref{lem:relaxed_end}. As in that case, it suffices to prove that
\begin{equation}
	\label{eq:compactdiff}
	\frac{q_{\ell,m-1,2n}}{m}-\frac{q_{\ell,m+1,2n}}{m+2}\geq0,
\end{equation}
for all $m\geq1$. We proceed by reverse induction on $\ell$, with base case $\ell=2n$. For the inductive step, note that $q$ satisfies the following recurrence for $\ell<2n$:
\begin{align*}
	\left\{
		\begin{array}{rlrl}
		q_{\ell,m,2n}&=0, & \text{ for } &m<0,\\
		q_{\ell,m,2n}&=\frac{\ell-m}{\ell-m+2}q_{\ell+1,m-1,2n}+\frac{\ell-m+2}{\ell+m+2}q_{\ell+1,m+1,2n}&&\\
		&~~~+\frac{2}{\ell-m+4}q_{\ell+2,m-2,2n}+\frac{2}{\ell+m+2}q_{\ell+3,m-1,2n}&&\\
		&~~~+\frac{4}{(\ell+m+4)(\ell+m+2)}q_{\ell+3,m+1,2n}, & \text{ for } & m\geq0.
		\end{array}
	\right.
\end{align*}
Now in order to prove \eqref{eq:compactdiff}, we expand the left-hand side $L(\ell,m,n)$ using the recurrence relation above. For $m\geq 2$, we use the inductive assumption, which says that \eqref{eq:compactdiff} holds for all larger values of $\ell$ and all $m$, to show that
\[L(\ell,m,n)\geq R_{1}(\ell,m)q_{\ell+1,m,2n}+R_{2}(\ell,m)q_{\ell+2,m-1,2n}+R_{3}(\ell,m)q_{\ell+3,m-2,2n},\]
for some explicit rational functions $R_{1}$, $R_{2}$ and $R_{3}$. Due to the nature of the functions $R_{1}$, $R_{2}$ and $R_{3}$, we can prove that the right-hand side above is positive using the inequalities
\[q_{\ell+1,m,2n}\geq \frac{\ell-m+1}{\ell-m+3}q_{\ell+2,m-1,2n}~~~~~\text{and}~~~~~q_{\ell+2,m-1,2n}\geq \frac{\ell-m+3}{\ell-m+5}q_{\ell+3,m-2,2n},\]
which follow directly from the recurrence relation. The case $m=1$ is similar, though we instead find
\[L(\ell,1,n)\geq \tilde{R}_{1}(\ell)q_{\ell+1,1,2n}+\tilde{R}_{2}(\ell)q_{\ell+2,0,2n}+\tilde{R}_{3}(\ell)q_{\ell+3,1,2n},\]
and we prove that the right hand side is positive using the inequalities
\[q_{\ell+1,1,2n}\geq \frac{\ell}{\ell+2}q_{\ell+2,0,2n}~~~~~\text{and}~~~~~q_{\ell+2,0,2n}\geq q_{\ell+3,1,2n},\]
which follow directly from the recurrence relation. We refer the accompanying worksheet~\cite{Wallner2019web} for more details. 
\end{proof}

Now, among the $\csr_{2n,0}$ weighted paths starting at $(0,0)$ and ending at $(2n,0)$, the proportion of those passing through some point $(2x,2y)$ is
\[\frac{\csr_{2x,2y}q_{2x,2y,2n}}{\csr_{2n,0}}\leq \frac{\csr_{2x,2y}q_{2x,2y,2n}}{\csr_{2x,0}q_{2x,0,2n}}\leq(2y+1)\frac{\csr_{2x,2y}}{\csr_{2x,0}}\leq(2y+1)\frac{\rs_{2x,2y}}{\cs_{2x,0}}\leq\frac{2y+1}{\gamma 4^{x}e^{3a_1x^{1/3}}x^{3/4}}{\binom{2x}{x+y}}.\]
We used the fact that $\hat{e}_{2x,2y} \leq d_{2x,2y}$ which we proved inductively using Lemma~\ref{lem:recurrence_bounds}, and we also used the lower bound~\eqref{eq:cnlower} for $e_{2x,0}$.  
We can finish in exactly the same way as in Lemma~\ref{lem:choose_const} for relaxed trees, thereby showing that there is some choice for $N$ such that $\csr_{2n,0}\leq 2\csrr_{2n,0}$ for all $n$. 

Recall that $\cs_{2n,0}\leq\csr_{2n,0}$ and there is some constant $\kappa_{1}$ such that
$\csrr_{n,m}\leq\kappa_{1}\hat{h}_{n}\hat{Y}_{n,m}.$ This implies that
\[c_{n}=n!c_{2n,0}\leq2\kappa_{1}n!\hat{h}_{2n}\hat{Y}_{2n,0}.\]
The right-hand side behaves asymptotically like $\Theta(n!4^{n}e^{3a_1n^{1/3}}n^{3/4})$, hence there is some constant $\gamma''$ such that
\[c_{n}\leq \gamma'' n!4^{n}e^{3a_1n^{1/3}}n^{3/4},\]
for all $n$. This completes the upper bound. Indeed, since we have now proven both the upper and lower bounds, which differ only in the constant term, they imply that
\[c_{n}= \Theta\left(n!4^{n}e^{3a_1n^{1/3}}n^{3/4}\right).\]


\section{Minimal finite automata}
\label{sec:automata}
In this section we use the results on compacted and relaxed trees to give bounds on the enumeration of a certain class of deterministic finite automata considered in \cite{DomaratzkiKismaShallit2002DFA,liskovets2006exact, domaratzki2006enumeration}. We start with some basic definitions of automata.
 
A \emph{deterministic finite automaton} (DFA) $A$ is a $5$-tuple $(\Sigma, Q, \delta,q_0,F)$, where $\Sigma$ is a finite set of letters called the \emph{alphabet}, $Q$ is a finite set of states, $\delta : Q \times \Sigma \to Q$ is the \emph{transition function}, $q_0$ is the \emph{initial state}, and $F \subseteq Q$ is the set of \emph{final states} (sometimes called \emph{accept states}). A DFA can be represented by a directed graph with one vertex $v_{s}$ for each state $s\in Q$, with the vertices corresponding to final states being highlighted, and for every transition $\delta(s,w)=\hat{s}$, there is an edge from $s$ to $\hat{s}$ labeled $w$ (see Figure \ref{fig:DFA}).

\begin{figure}[ht]
	\centering
	\includegraphics[scale=0.9]{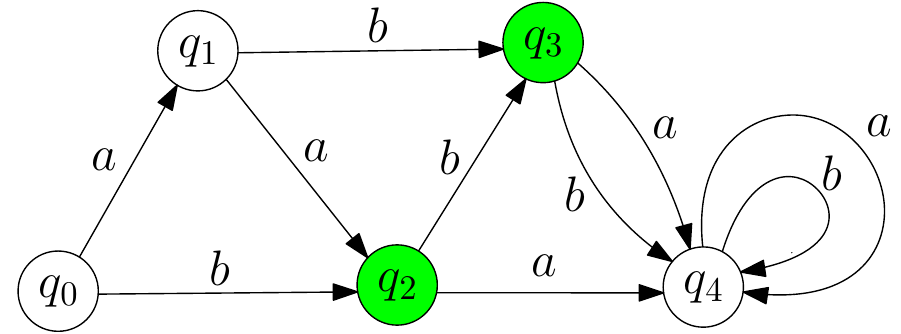}%
	\caption{The unique minimal DFA with $5$ states for the language $\{aa,aab,ab,b,bb\}$. Here, $q_{0}$ is the initial state, $q_{2}$ and $q_{3}$ are the final states, and $q_{4}$ is the unique dead state.}
	\label{fig:DFA}
\end{figure}	

 A word $w = w_1 w_2 \cdots w_n \in \Sigma^*$ is \emph{accepted} by $A$ if the sequence of states $(s_0,s_1,\ldots,s_n) \in Q^{n+1}$ defined by $s_0=q_0$ and $s_{i+1} = \delta(s_i,w_i)$ for $i=0,\ldots,n-1$ ends with $s_n \in F$ a final state.
The set of words accepted by $A$ is called the \emph{language} $\Lc(A)$ recognized by $A$.
It is well-known that DFAs recognize exactly the set of regular languages.
Note that every DFA recognizes a unique language, but a language can be recognized by several different DFAs. 
A DFA is called \emph{minimal} if no DFA with fewer states recognizes the same language.
The Myhill-Nerode Theorem states that every regular language is recognized by a unique minimal DFA (up to isomorphism)~\cite[Theorem~3.10]{HopcroftUllman1979Automata}.
For more details on automata see~\cite{HopcroftUllman1979Automata}.

Since every regular language defines a unique minimal automaton, one can define the complexity of the language to be the number of states in the corresponding automaton. More precisely, this is an interpretation for the space complexity of the language. 

The asymptotic proportion of minimal DFAs in the class of (not necessarily acyclic) automata was solved by Bassino, Nicaud, and Sportiello in~\cite{BassinoEtal2012Automata}, building on enumeration results by Korshunov~\cite{Korshunov1986Automata,Korshunov1978Automata} and Bassino and Nicaud~\cite{BassinoNicaud2007Automata}.
This result also exploits an underlying tree structure of the automata, but this tree structure comes from a different traversal than what we use.
In that case, no stretched exponential appears in the asymptotic enumeration, and
the minimal automata account for a constant fraction of all automata.

The analogous problem in the restricted case of automata that recognize a finite language is widely open (see for example \cite{domaratzki2006enumeration}). This corresponds to counting finite languages by their space complexity. To show the relation between these automata and compacted and relaxed trees, we need the following lemma from \cite[Lemma~2.3]{liskovets2006exact} or \cite[Section~3.4]{HopcroftUllman1979Automata}. For the convenience of the reader, we include a proof of one direction here. 
\begin{lemma}\label{lem:DFAmin}A DFA $A$ is the minimal automaton for some finite language if and only if it has the following properties:
\begin{itemize}
\item There is a unique \emph{sink} $s$. That is, a state which is not a final state and with all transitions from $s$ end at $s$ that is, $\delta(s,w)=s$.
\item $A$ is \emph{acyclic}: the corresponding directed graph has no cycles except for the loops at the dead state.
\item $A$ is \emph{initially connected}: for any state $p$ there exists a word $w \in \Sigma^*$ such that $A$ reaches the state $p$ upon reading $w$.
\item $A$ is \emph{reduced}: for any two different states $q$, $q'$, the two automata with initial state $q$ and $q'$ recognize different languages. 
\end{itemize}
\end{lemma}
\begin{proof}
We will show one direction of this proof: that a minimal automaton has the four properties. For a proof of the reverse direction see e.g.~\cite[Lemma~2.3]{liskovets2006exact} or \cite[Section~3.4]{HopcroftUllman1979Automata}.

If $A$ is minimal but not reduced then there are two states $q$ and $q'$ from which the same language is recognized. These two states can be merged into a single state without changing the language, contradicting the minimality of $A$. This implies that there is at most one state $q$ from which the empty language is recognized. Moreover, such a state must exist for the language to be finite. This state $q$ must therefore be the unique sink.

The fact the $A$ is acyclic follows immediately from the fact that $A$ recognizes a finite language. Finally, if we remove from $A$ all states $p$ that cannot be reached, the language accepted by the automaton will not be changed, so by the minimality of $A$, there must be no such states and $A$ must be initially connected.
\end{proof}

We note here one consequence of this lemma: since the automaton is acyclic, there must be some state $q$ other than the sink $s$ such that all transitions from $q$ end at $s$. Then since the automaton is reduced, there must be only one such state $q$, and it must be an accept state.

Now using our asymptotic results on compacted and relaxed trees, we give the following new bounds on the asymptotic number of such automata, determining their asymptotics up to a polynomial multiplicative term.
\begin{theo}
	\label{theo:m2nbounds}
	Let $m_{2,n}$ be the number of minimal DFAs over a binary alphabet with $n$ transient states (and a unique sink) that recognize a finite language. 
	Then, 
	\begin{align*}
		2^{n-1} c_n \leq m_{2,n} \leq 2^{n-1} r_n.
	\end{align*}
	As a consequence, there exist positive constants $\kappa_1$ and $\kappa_2$ such that
	\begin{align*}
		\kappa_1 n^{3/4} 
			\leq \frac{m_{2,n}}{n! 8^n e^{3a_1n^{1/3}}} 
			\leq \kappa_2 n.
	\end{align*}
\end{theo}

\begin{proof}
	By the lemma above, a compacted tree $\mathcal{C}$ can be transformed into an automaton $A$ that recognizes a finite language over the alphabet $\{a, b\}$ as follows:
	The states of the automaton $A$ correspond to the internal nodes and the leaf of $\mathcal{C}$. The initial state corresponds to the root, and at each state, the transition after reading $a$ (resp.~$b$) is given by the left (resp.~right) child or pointer in $\mathcal{C}$. The leaf is designated as the unique sink, and we can choose any subset of internal nodes as final states, with the condition that the unique node with two pointers to the sink is always a final state.
	
	To prove the minimality of such automata, we just need to check the four conditions of Lemma \ref{lem:DFAmin}. The fact that $A$ is acyclic and has a unique sink follow immediately from the fact the $\mathcal{C}$ is a DAG. Then $A$ is initially connected because $\mathcal{C}$ has a unique source. 
	Now we will show that $A$ is reduced. 
	For any state $q$ in $A$, let $\Lc(q)$ denote the language recognized by the automaton with initial state $q$. 
	Now suppose that $A$ is not reduced and let $q$ and $q'$ be different states of $A$ satisfying $\Lc(q)=\Lc(q')$. 
	We also assume that amongst all such pairs $(q,q')$, the length of the longest word in $\Lc(q)$ is minimized. 
	Since the unique node with both pointers to the sink is a final state, the language $\Lc(q)$ can only be empty if $q$ and $q'$ are both the final state, which is impossible. 
	Since $\Lc(q)=\Lc(q')$ we must have $\Lc(\delta(q,a))=\Lc(\delta(q',a))$ and $\Lc(\delta(q,b))=\Lc(\delta(q',b))$. 
	Then, the minimality condition on the language $\Lc(q)$ implies that $\delta(q,a)=\delta(q',a)$ and $\delta(q,b)=\delta(q',b)$. 
	But this means that the node $u$ and $v$ in $\mathcal{C}$ corresponding to $q$ and $q'$ have the same left child and the same right child, contradicting the fact that $\mathcal{C}$ is compacted. 
	This completes the proof that $A$ is reduced.
	
	Hence, each of the $2^{n-1}$ subsets of the remaining states (not the sink and not the node with two pointers to the sink) gives a valid minimal automaton of size $n$.
	
	Applying the same construction to relaxed trees gives an upper bound, as every minimal automaton, after forgetting which states are final, corresponds by construction to a relaxed tree. Note that this observation has already been made in~\cite[Equation~(1)]{liskovets2006exact}, yet the asymptotics was not known.  
\end{proof}

Using the methods of the present work, the authors showed in a companion paper~\cite{ElveypriceFangWallner2020DFA} that 
\begin{align*}
	m_{2,n} = \Theta\left( n! \, 8^n e^{3a_1n^{1/3}} n^{7/8} \right).
\end{align*}
To our knowledge, our results give the best known bounds on the asymptotic number of minimal DFAs on a binary alphabet recognizing a finite language. 
Note that they disprove the conjecture $m_{2,n} \sim K \, 2^{n-1}r_n$ for some $K>0$ of Liskovets based on numerical data; see~\cite[Equation~(16)]{liskovets2006exact}. 
Previously, Domaratzki derived in~\cite{Domaratzki2004Bounds} the lower bound 
\[m_{2,n} \geq \frac{(2n-1)!}{(n-1)!}c_1^{n-1},\] 
with $c_1 \approx 1.0669467$, which implies the asymptotic bound $m_{2,n} \geq  \frac{n!(4c_1)^n}{2c_1\sqrt{\pi n}}$ 
(note that $m_{2,n} = f_2'(n+1)$ in his results).
Furthermore, Domaratzki showed in~\cite{Domaratzki2004Genocchi} the upper bound 
\[m_{2,n} \leq 2^{n-1}G_{2n+2},\] 
where $G_{2n}$ are the Genocchi numbers defined by $\frac{2t}{e^t+1} = t + \sum_{n \geq 1}(-1)^n G_{2n} \frac{t^{2n}}{(2n)!}$. This gives the asymptotic bound $m_{2,n} \leq 4 (2n)! (\frac{2}{\pi^2})^{n+1} n^2.$ This bound, however, is much larger than the superexponential growth given by $n!$ in our upper bound. 
While not explicitly formulated in the literature, it is possible to bound the acyclic DFAs by general DFAs using the results by Korshunov~\cite{Korshunov1986Automata,Korshunov1978Automata} (see also~\cite[Theorem~18]{BassinoNicaud2007Automata}). 
Thereby, we get the upper bound 
\[m_{2,n} = \LandauO\left(n! (2e^2\nu)^n\right),\] 
where $\nu = \alpha^\alpha (1+\alpha)^{1-\alpha} \approx 0.8359$ with $\alpha$ being the solution of $1+x=xe^{2/(1+x)}$, and therefore $2e^2\nu \approx 12.3531$, which is significantly larger than the exponential growth in our upper bound.


\section*{Acknowledgments}
We would like to thank Tony Guttmann for sending us calculations on the asymptotic form of pushed Dyck paths. More generally, we thank him, Cyril Banderier and Andrea Sportiello for interesting discussions on the presence of a stretched exponential. We want to thank the anonymous referees for their detailed comments which improved the presentation of this work.

\addcontentsline{toc}{section}{References}
\bibliographystyle{abbrv}
\bibliography{Bibliography}
\label{sec:biblio}

\Addresses

\end{document}